\documentclass[a4paper,10pt]{article}
\usepackage{mathtools}
\usepackage{amsfonts}
\usepackage{amsthm}
\usepackage{graphicx}
\usepackage{framed}
\usepackage{float}
\usepackage{upgreek}
\usepackage{enumerate}
\usepackage{MnSymbol}
\usepackage{amsbsy}
\usepackage{fullpage}
\usepackage{bbm}
\usepackage{soul}
\usepackage{color}
\usepackage{microtype}
\PassOptionsToPackage{hyphens}{url}
\usepackage{hyperref}

\theoremstyle{plain}
\newtheorem{thm}{Theorem}[section]
\newtheorem{lemma}[thm]{Lemma}
\newtheorem{cor}[thm]{Corollary}

\newtheorem{prop}[thm]{Proposition}

\theoremstyle{definition}
\newtheorem{defi}[thm]{Definition}
\newtheorem{rmk}[thm]{Remark}
\newtheorem{ex}[thm]{Example}

\theoremstyle{remark}
\begin{document}

\title{Classification of random circle homeomorphisms\\ up to topological conjugacy}
\author{Thai Son Doan\footnote{Institute of Mathematics, Vietnam Academy of Science and Technology, 18 Hoang Quoc Viet, Ha Noi, Vietnam}, Jeroen S.W.~Lamb\footnote{Department of Mathematics, Imperial College London, 180 Queen's Gate, London SW7 2AZ, United Kingdom}, Julian Newman\footnote{Faculty of Mathematics, University of Bielefeld, 33615 Bielefeld, Germany}, Martin Rasmussen\footnote{Department of Mathematics, Imperial College London, 180 Queen's Gate, London SW7 2AZ, United Kingdom}}

\maketitle 

\begin{abstract}

\noindent We provide a classification of random orientation-preserving homeomorphisms of $\mathbb{S}^1$, up to topological conjugacy of the random dynamical systems generated by i.i.d.\ iterates of the random homeomorphism. This classification covers all random circle homeomorphisms for which the noise space is a connected Polish space and an additional extremely weak condition is satisfied.

\end{abstract}

\section{Introduction}

In classical dynamical systems theory, a common question to ask is whether, for a given pair of self-maps $(f,g)$ of some compact metric space, there is a topological conjugacy from $f$ to $g$. In this paper, we consider the ``analogous'' question for a \emph{noisy pair of maps} $\{ \, (f_\alpha,g_\alpha) \, \}_{\alpha \in \Delta}$ where $\alpha$ is drawn randomly from some probability space $(\Delta,\mathcal{B}(\Delta),\nu)$. In our case, a random map $(f_\alpha)_{\alpha \in \Delta}$ is viewed ``dynamically'' by considering \emph{i.i.d.\ iterations}; topological conjugacy is then understood in the ``random dynamical systems'' framework, namely as a topology-preserving cohomology between the cocycles generated by the random maps $(f_\alpha)_{\alpha \in \Delta}$ and $(g_\alpha)_{\alpha \in \Delta}$ over the shift map on $(\Delta^\mathbb{Z},\mathcal{B}(\Delta)^{\otimes \mathbb{Z}},\nu^{\otimes \mathbb{Z}})$.
\\ \\
Topological conjugacy of random dynamical systems has been considered before, most notably in \cite{IL02}, which provides weak conditions under which a Wiener-driven stochastic differential equation can be transformed to a random differential equation. \cite{cdlr16} shows that a breakdown of uniform topological equivalence characterises a stochastic pitchfork bifurcation. \cite{Gundlach1995ChaosIR} describes cases in which ``chaotic'' random dynamical systems can be measurably or topologically conjugated to a random shift. For results on \emph{local} conjugacy, see \cite{Li20163191} and references therein.
\\ \\
This present paper is the first study seeking to \emph{classify} a broad class of random dynamical systems up to topological conjugacy; hence, we naturally focus on the simplest case---which already turns out to be remarkably subtle---namely random circle homeomorphisms. Under reasonable conditions, we obtain a complete classification of random circle homeomorphisms up to orientation-preserving conjugacy, and hence up to topological conjugacy.
\\ \\
The structure of the paper is as follows: In Section~2, we give the necessary preliminaries, state our main result (Theorem~\ref{MAIN}), and present some examples. In Section~3, we develop further preliminary results needed for the proof of our main result, particularly regarding the random invariant objects that are key to establishing the existence or non-existence of a topological conjugacy. In Section~4, we prove our main result.

\section{Results and examples}

\subsection{Random circle homeomorphisms}

Fix a connected Polish space $\Delta$, equipped with a Borel probability measure $\nu$ of full support.

\begin{defi}
A \emph{random circle homeomorphism} is a $\Delta$-indexed family $\mathbf{f}=(f_\alpha)_{\alpha \in \Delta}$ of orientation-preserving homeomorphisms $f_\alpha \in \mathrm{Homeo}^+(\mathbb{S}^1)$ such that
\begin{enumerate}[\indent (i)]
\item $(\alpha,x) \mapsto f_\alpha(x)$ is continuous;
\item for any non-dense bi-infinite sequence $(x_n)_{n \in \mathbb{Z}}$ in $\mathbb{S}^1$, there exists $n \in \mathbb{Z}$ and $\alpha \in \Delta$ such that $f_\alpha(x_n) \neq x_{n+1}$.
\end{enumerate}
\end{defi}

\noindent Condition~(ii) is a very weak ``non-degeneracy'' condition that guarantees that a random circle homeomorphism has the structure described in Proposition~\ref{permute}.

\begin{defi}
A \emph{symmetry} of a random circle homeomorphism $\mathbf{f}$ is an orientation-preserving homeomorphism $\tau \in \mathrm{Homeo}^+(\mathbb{S}^1) \setminus \{\mathrm{id}_{\mathbb{S}^1}\}$ such that $\tau^k=\mathrm{id}_{\mathbb{S}^1}$ for some $k \in \mathbb{N}$ and $f_\alpha$ commutes with $\tau$ for all $\alpha$. The smallest $k$ such that $\tau^k=\mathrm{id}_{\mathbb{S}^1}$ is called the \emph{order} of $\tau$.
\end{defi}

\begin{defi} \label{minimal}
Given a random circle homeomorphism $\mathbf{f}$, a closed set $G \subset \mathbb{S}^1$ is said to be \emph{$\mathbf{f}$-invariant} if $f_\alpha(G) \subset G$ for every $\alpha \in \Delta$. An \emph{$\mathbf{f}$-minimal set} is a non-empty $\mathbf{f}$-invariant closed set $G \subset \mathbb{S}^1$ containing no non-empty $\mathbf{f}$-invariant closed proper subset.
\end{defi}

\noindent Note that any two distinct $\mathbf{f}$-minimal sets are mutually disjoint. It is well-known that every $\mathbf{f}$-invariant closed set contains at least one $\mathbf{f}$-minimal set.

\begin{prop} \label{permute}
A random circle homeomorphism $\mathbf{f}$ admits finitely many minimal sets. Letting $G$ be the union of all the $\mathbf{f}$-minimal sets, the number $k_\mathbf{f}$ of connected components of $G$ is finite, and there exists $l_\mathbf{f} \in \{0,\ldots,k_\mathbf{f}-1\}$ such that for every $\alpha \in \Delta$, for each $i \in \{0,\ldots,k_\mathbf{f}-1\}$,
\[ f_\alpha(G_i) \ \subset \ G_{i+l_\mathbf{f} \; \mathrm{mod} \; k_\mathbf{f}} \]
where $G_0,\ldots,G_{k_\mathbf{f}-1}$ are the connected components of $G$ going anticlockwise round the circle.
\end{prop}

\noindent We will prove Proposition~\ref{permute} as a consequence of \cite[Theorem~B]{Mal14} in Section~\ref{permproof}.

\begin{rmk} \label{permrem}
Letting $p \in \{1,\ldots,k_\mathbf{f}\}$ be the highest common factor of $k_\mathbf{f}$ and $l_\mathbf{f}$, and letting $q=\frac{k_\mathbf{f}}{p}$, we have that each $\mathbf{f}$-minimal set takes the form $\bigcup_{j=0}^{q-1} G_{i + jp}$ for some $i \in \{0,\ldots,p-1\}$. No connected component of $G$ is a singleton; when $k_\mathbf{f}=1$, the $\mathbf{f}$-minimal set $G$ could be \emph{either} an arc \emph{or} the whole circle. In the case that $k_\mathbf{f}=1$ and $G$ is an arc, $\mathbf{f}$ cannot admit a symmetry.
\end{rmk}

\begin{rmk} \label{ergmin}
A \emph{stationary measure} of a random circle homeomorphism $\mathbf{f}$ is a Borel probability measure $\rho$ on $\mathbb{S}^1$ such that the pushforward of $\nu \otimes \rho$ under $(\alpha,x) \mapsto f_\alpha(x)$ is equal to $\rho$; an extreme point of the convex set of all $\mathbf{f}$-stationary measures is called an \emph{ergodic measure} of $\mathbf{f}$. By \cite[Theorem~B]{Mal14}, the support of every $\mathbf{f}$-ergodic is $\mathbf{f}$-minimal, and each $\mathbf{f}$-minimal set is precisely equal to the support of exactly one $\mathbf{f}$-ergodic measure; the $\mathbf{f}$-stationary measures are simply the convex combinations of the $\mathbf{f}$-ergodic measures. So every $\mathbf{f}$-stationary measure is supported on a union of $\mathbf{f}$-minimal sets; moreover, one can show that every $\mathbf{f}$-stationary measure is atomless.
\end{rmk}

\subsection{Main result} \label{Main result}

Suppose we have two random circle homeomorphisms $\mathbf{f}=(f_\alpha)_{\alpha \in \Delta}$ and $\mathbf{g}=(g_\alpha)_{\alpha \in \Delta}$.
\\ \\
We say that \emph{$\mathbf{f}$ and $\mathbf{g}$ are deterministically topologically conjugate} if there exists a homeomorphism $h \in \mathrm{Homeo}(\mathbb{S}^1)$ such that $f_\alpha = h^{-1} \circ g_\alpha \circ h$ for all $\alpha$; if $h$ can be chosen to be orientation-preserving, then we say that \emph{$\mathbf{f}$ and $\mathbf{g}$ are deterministically orientationally conjugate}. (So a symmetry of $\mathbf{f}$ is just a homeomorphism of non-trivial finite order via which $\mathbf{f}$ is deterministically orientationally conjugate to itself.)
\\ \\
Let $(\Omega,\mathcal{F},\mathbb{P})=(\Delta^\mathbb{Z},\mathcal{B}(\Delta)^{\otimes \mathbb{Z}},\nu^{\otimes \mathbb{Z}})$, and define $\theta \colon \Omega \to \Omega$ by $\theta\left( (\alpha_n)_{n \in \mathbb{Z}} \right)=(\alpha_{n+1})_{n \in \mathbb{Z}}$. We say that \emph{$\mathbf{f}$ and $\mathbf{g}$ have topologically conjugate dynamics} if there exists an $\Omega$-indexed family $(h_\omega)_{\omega \in \Omega}$ of homeomorphisms $h_\omega \in \mathrm{Homeo}(\mathbb{S}^1)$ such that $\omega \mapsto h_\omega(x)$ is measurable for each $x \in \mathbb{S}^1$, and
\begin{equation}\label{conjugation} f_{\alpha_0} \ = \ h_{\theta\omega}^{-1} \circ g_{\alpha_0} \circ h_\omega \end{equation}
for $\mathbb{P}$-almost every $\omega \!=\! (\alpha_n)_{n \in \mathbb{Z}} \in \Omega$. Given such a family of homeomorphisms $(h_\omega)$, since $\mathbb{P}$ is ergodic with respect to $\theta$ we have that either $h_\omega$ is orientation-preserving for $\mathbb{P}$-almost all $\omega$, or $h_\omega$ is orientation-reversing for $\mathbb{P}$-almost all $\omega$; if $(h_\omega)$ can be chosen such that the former holds, then we say that \emph{$\mathbf{f}$ and $\mathbf{g}$ have orientationally conjugate dynamics}.

\begin{rmk} \label{top op}
Viewing $\mathbb{S}^1$ as the circle group $\mathbb{R}/\mathbb{Z}$: $\mathbf{f}$ and $\mathbf{g}$ have topologically conjugate dynamics if and only either
\begin{enumerate}[\indent (i)]
\item $\mathbf{f}$ and $\mathbf{g}$ have orientationally conjugate dynamics, or
\item $\mathbf{f}$ and $\bar{\mathbf{g}}$ have orientationally conjugate dynamics,
\end{enumerate}
where $\bar{\mathbf{g}}$ is the ``mirror-reversed version of $\mathbf{g}$'' given by $\bar{g}_\alpha(x)=-g_\alpha(-x)$. The corresponding statements for deterministic conjugacy also hold.
\end{rmk}

\noindent Our main theorem is the following:

\begin{thm} \label{MAIN}
Let $\mathbf{f}$ and $\mathbf{g}$ be random circle homeomorphisms, and let $k_\mathbf{f}$, $l_\mathbf{f}$, $k_\mathbf{g}$ and $l_\mathbf{g}$ be as in Proposition~\ref{permute}.

$\mathbf{f}$ and $\mathbf{g}$ have orientationally conjugate dynamics if and only if one of the following four statements holds:
\begin{enumerate}[\indent (a)]
\item $k_\mathbf{f}=k_\mathbf{g} \geq 2$, and $l_\mathbf{f}=l_\mathbf{g}$;
\item $k_\mathbf{f}=k_\mathbf{g}=1$, and neither $\mathbf{f}$ nor $\mathbf{g}$ admits a symmetry;
\item $k_\mathbf{f}=k_\mathbf{g}=1$, $\mathbf{f}$ admits a symmetry of order at least $3$, and $\mathbf{f}$ and $\mathbf{g}$ are deterministically orientationally conjugate;
\item $k_\mathbf{f}=k_\mathbf{g}=1$, $\mathbf{f}$ admits an order-$2$ symmetry but no higher-order symmetry, and $\mathbf{f}$ and $\mathbf{g}$ are deterministically \emph{topologically} conjugate.
\end{enumerate}
Hence, $\mathbf{f}$ and $\mathbf{g}$ have topologically conjugate dynamics if and only if one of the following three statements holds:
\begin{enumerate}[\indent (a')]
\item $k_\mathbf{f}=k_\mathbf{g} \geq 2$, and $l_\mathbf{g} \in \{l_\mathbf{f},k_\mathbf{f}-l_\mathbf{f}\}$;
\item $k_\mathbf{f}=k_\mathbf{g}=1$, and neither $\mathbf{f}$ nor $\mathbf{g}$ admits a symmetry;
\item $k_\mathbf{f}=k_\mathbf{g}=1$, $\mathbf{f}$ admits a symmetry, and $\mathbf{f}$ and $\mathbf{g}$ are deterministically topologically conjugate.
\end{enumerate}
\end{thm}

\noindent Note that in case~(c)/(d)/(c'), $\mathbb{S}^1$ is both $\mathbf{f}$-minimal and $\mathbf{g}$-minimal. However, in case~(b)/(b'), both the $\mathbf{f}$-minimal set and the $\mathbf{g}$-minimal set could be either an arc or the whole circle $\mathbb{S}^1$.
\\ \\
Let us now briefly describe the dynamics in the different cases: As in the Introduction, a random circle homeomorphism $\mathbf{f}$ is studied ``dynamically'' by considering i.i.d.\ iterates. Let $G$ the union of the $\mathbf{f}$-minimal sets, as in Proposition~\ref{permute}.
\begin{itemize}
\item For $\mathbf{f}$ as in case~(a)/(a'), there is almost surely a $k_\mathbf{f}$-element set $\mathcal{R}$, consisting of exactly one point in each connected component of $\mathbb{S}^1 \setminus G$, such that every compact connected subset of $\mathbb{S}^1 \setminus \mathcal{R}$ contracts in diameter to $0$ at an exponential rate under the iterates of $\mathbf{f}$. (This essentially follows from results of \cite{Mal14}, together with Lemma~\ref{Li}/\ref{pull} of this paper.)
\item For $\mathbf{f}$ as in case~(b)/(b'), there is almost surely a point $r \in \mathbb{S}^1$ such that every compact subset of $\mathbb{S}^1 \setminus \{r\}$ contracts in diameter to $0$ at an exponential rate under the iterates of $\mathbf{f}$; in the case that the minimal set $G$ is a proper subset of $\mathbb{S}^1$, the point $r$ lies in $\mathbb{S}^1 \setminus G$. (See \cite{Mal14} and \cite{Ant84}.)
\item For $\mathbf{f}$ as in case~(c)/(d)/(c'), either:
\begin{enumerate}[\indent (i)]
\item $\mathbf{f}$ is deterministically orientationally conjugate to a rational rotation (in which case $\mathbf{f}$ admits symmetries of all finite orders); or
\item there is a symmetry $\tau$ of maximal order $m \in \mathbb{N}$, and there is almost surely an orbit $\mathcal{R}$ of $\tau$ such that every compact connected subset of $\mathbb{S}^1 \setminus \mathcal{R}$ contracts in diameter to $0$ at an exponential rate under the iterates of $\mathbf{f}$. (Case~(d) is precisely the situation that $m=2$.)
\end{enumerate}
(Again, see \cite{Mal14} and \cite{Ant84}.)
\end{itemize}

\subsection{Examples}

Identifying $\mathbb{S}^1 \equiv \mathbb{R}/\mathbb{Z}$, we write $[x] \in \mathbb{S}^1$ for the projection of $x \in \mathbb{R}$. The first two examples we consider will be noisy perturbations of maps taking the form
\[ f([x]) \ = \ \left[ x \, + \; \tfrac{1}{2\pi k}\sin(2\pi kx)+\tfrac{l}{k} \right] \]
for some $k \in \mathbb{N}$ and $l \in \{0,\ldots,k-1\}$. The set $\mathcal{R}_f:=\{[\frac{i}{k}] : i \in \{0,\ldots,k-1\} \}$ is a repelling invariant set of $f$, and all trajectories of $f$ not starting in $\mathcal{R}_f$ tend towards the attracting invariant set $\mathcal{A}_f:=\{[\frac{2i+1}{2k}] : i \in \{0,\ldots,k-1\} \}$. On both of these invariant sets, $f$ coincides with the rotation $[x] \mapsto [x+\frac{l}{k}]$. The rotations $[x] \mapsto [x+\frac{i}{k}]$, $0 \leq i \leq k-1$, commute with $f$; moreover, for any orientation-preserving homeomorphism $\tau$ commuting with $f$, we have that $\mathcal{A}_f$ and $\mathcal{R}_f$ are $\tau$-invariant, and so if $\tau$ has finite order then $\tau^m=\mathrm{id}_{\mathbb{S}^1}$. Given a map
\[ g([x]) \ = \ \left[ x \, + \; \tfrac{1}{2\pi k'}\sin(2\pi k'x)+\tfrac{l'}{k'} \right] \]
where $k' \in \mathbb{N}$ and $l' \in \{0,\ldots,k'-1\}$, if $f \neq g$ (i.e.\ if either $k \neq k'$ or $l \neq l'$) then there is no orientation-preserving topological conjugacy (in the classical sense) from $f$ to $g$.

\begin{ex} \label{indepnoise}
Let $\Delta=[-1,1] \times [-1,1]$, with $\nu$ the normalised Lebesgue measure, and define the random circle homeomorphisms $\mathbf{f}$ and $\mathbf{g}$ by
\begin{align*}
f_{(\alpha^1\!,\alpha^2)}([x]) \ &= \ \left[ x \, + \; \tfrac{1}{2\pi k}\sin(2\pi kx)+\tfrac{l}{k} \ + r\alpha^1 \right] \\
g_{(\alpha^1\!,\alpha^2)}([x]) \ &= \ \left[ x \, + \; \tfrac{1}{2\pi k'}\sin(2\pi k'x)+\tfrac{l'}{k'} \ + r'\alpha^2 \right]
\end{align*}
where $k,k' \in \mathbb{N}$, $l \in \{0,\ldots,k-1\}$, $l' \in \{0,\ldots,k'-1\}$, and $r,r'>0$. If $k \geq 2$ then the rotations $[x] \mapsto [x+\frac{i}{k}]$, $1 \leq i \leq k-1$, are symmetries of $\mathbf{f}$. One can check that:
\begin{enumerate}[\indent (i)]
\item if $r \leq \frac{1}{2\pi k}$, then the union of all the $\mathbf{f}$-minimal sets is a proper subset of $\mathbb{S}^1$ with $k$ connected components, each of which is an arc about a point in $\mathcal{A}_f$;
\item if $r > \frac{1}{2\pi k}$, then the whole circle $\mathbb{S}^1$ is $\mathbf{f}$-minimal.
\end{enumerate}
Obviously, the corresponding statements also hold for $\mathbf{g}$. Applying Theorem~\ref{MAIN}, we have the following:\emph{
\begin{itemize}
\item if either $k \neq k'$ or $l \neq l'$, then $\mathbf{f}$ and $\mathbf{g}$ do not have orientationally conjugate dynamics;
\item if $k=k'=1$, then $\mathbf{f}$ and $\mathbf{g}$ have orientationally conjugate dynamics (\emph{regardless} of the values of $r$ and $r'$);
\item in the case that $k=k' \geq 2$ and $l=l'$, $\mathbf{f}$ and $\mathbf{g}$ have orientationally conjugate dynamics if and only if $\max(r,r') \leq \frac{1}{2\pi k}$.
\end{itemize}
}
\end{ex}

\noindent Our next example is similar to the above, the only difference being that the added noise terms for $\mathbf{f}$ and $\mathbf{g}$ are not independent, but instead are mirror-reversed:

\begin{ex} \label{mirrornoise}
Let $\Delta=[-1,1]$, with $\nu$ the normalised Lebesgue measure, and define the random circle homeomorphisms $\mathbf{f}$ and $\mathbf{g}$ by
\begin{align*}
f_\alpha([x]) \ &= \ \left[ x \, + \; \tfrac{1}{2\pi k}\sin(2\pi kx)+\tfrac{l}{k} \ + r\alpha \right] \\
g_\alpha([x]) \ &= \ \left[ x \, + \; \tfrac{1}{2\pi k'}\sin(2\pi k'x)+\tfrac{l'}{k'} \ - r\alpha \right]
\end{align*}
where $k,k' \in \mathbb{N}$, $l \in \{0,\ldots,k-1\}$, $l' \in \{0,\ldots,k'-1\}$, and $r>0$. Applying Theorem~\ref{MAIN}, we have the following:\emph{
\begin{itemize}
\item if either $k \neq k'$ or $l \neq l'$, then $\mathbf{f}$ and $\mathbf{g}$ do not have orientationally conjugate dynamics;
\item if either $k=k'=1$, or $k=k'=2$ and $l=l'$, then $\mathbf{f}$ and $\mathbf{g}$ have orientationally conjugate dynamics;
\item in the case that $k=k' \geq 3$ and $l=l'$, $\mathbf{f}$ and $\mathbf{g}$ have orientationally conjugate dynamics if and only if $r \leq \frac{1}{2\pi k}$.
\end{itemize}
}
\end{ex}

\noindent Now in the classical (deterministic) setting, two distinct rigid rotations $f$ and $g$ of the circle cannot be topologically conjugate; and moreover, no sufficiently $C^0$-small perturbation of $f$ can be topologically conjugate to a sufficiently small $C^0$-perturbation of $g$. However, arbitrarily small \emph{random} perturbations of any two distinct rigid rotations \emph{can} have topologically conjugate dynamics:

\begin{ex} \label{rotate}
Let $\Delta=[0,1]$, with $\nu$ the normalised Lebesgue measure, and define the random circle homeomorphisms $\mathbf{f}$ and $\mathbf{g}$ by
\begin{align*}
f_\alpha([x]) \ &= \ \left[ x \, + \, c \; + \; \varepsilon\sin(2\pi (x+\alpha)) \right] \\
g_\alpha([x]) \ &= \ \left[ x \, + c' \, + \; \varepsilon\sin(2\pi (x+\alpha)) \right]
\end{align*}
where $\varepsilon \in (0,\frac{1}{2\pi}]$ and $c,c' \in [0,1)$. As in the example in \cite{LeJan87} or \cite[Section~5]{newman_2017}, $\mathbb{S}^1$ is $\mathbf{f}$-minimal and $\mathbf{g}$-minimal, and $\mathbf{f}$ and $\mathbf{g}$ do not admit a symmetry. So by Theorem~\ref{MAIN}, $\mathbf{f}$ and $\mathbf{g}$ have orientationally conjugate dynamics.
\end{ex}

\section{Preparations for the proof of the main result}

\subsection{Preliminaries}

Throughout this paper, given an element $\omega$ of the sequence space $\Omega=I^\mathbb{Z}$, we write $\alpha_i$ for the $i$-th coordinate of $\omega$.
\\ \\
We identify $\mathbb{S}^1$ with the quotient of the group $(\mathbb{R},+)$ by its subgroup $\mathbb{Z}$, with $[\cdot]$ denoting the standard projection. Given $m \in \mathbb{N}$ and $x \in \mathbb{S}^1$, we write $mx \in \mathbb{S}^1$ for the $m$-fold sum of $x$. For any $x,y \in \mathbb{S}^1$, we write $[x,y] \subset \mathbb{S}^1$ for the projection of $[x',y'] \subset \mathbb{R}$ where $x'$ may be any lift of $x$ and $y'$ is the unique lift of $y$ in $[x',x'+1)$; we write $d_+(x,y) \in [0,1)$ for the Lebesgue measure of $[x,y]$, and we define $d(x,y):=\min(d_+(x,y),d_+(y,x))$. We write $[x,y[$ for $[x,y] \setminus \{y\}$; we write $]x,y]$ for $[x,y] \setminus \{x\}$; and we write $]x,y[$ for $[x,y] \setminus \{x,y\}$.\footnote{We use this style of notation, rather than round bracket notation, because later on, we will frequently be looking at points in $\mathbb{S}^1 \times \mathbb{S}^1$.} Given a non-empty connected non-dense subset $A$ of $\mathbb{S}^1$, define the points $\partial_- A$ and $\partial_+ A$ by $\bar{A} = [\partial_- A, \partial_+ A]$.
\\ \\
Given any probability space $(\tilde{\Omega},\tilde{\mathcal{F}},\tilde{\mathbb{P}})$ and a compact metric space $X$, we say that a measure $\mu$ on $\tilde{\Omega} \times X$ \emph{has $\tilde{\Omega}$-marginal $\tilde{\mathbb{P}}$} if $\mu(E \times X)=\tilde{\mathbb{P}}(E)$ for all $E \in \mathcal{F}$. We can identify such a measure $\mu$ with its \emph{disintegration} $(\mu_\omega)_{\omega \in \tilde{\Omega}}$, that is, the unique (up to $\mathbb{P}$-a.e.~equality) $\tilde{\Omega}$-indexed family of probability measures $\mu_\omega$ on $X$ such that $\omega \mapsto \mu_\omega(A)$ is measurable for all $A \in \mathcal{B}(X)$ and
\[ \mu(E \times A) \ = \ \int_E \, \mu_\omega(A) \, \tilde{\mathbb{P}}(d\omega) \]
for all $E \in \tilde{\mathcal{F}}$ and $A \in \mathcal{B}(X)$. Given a sub-$\sigma$-algebra $\tilde{\mathcal{G}}$ of $\tilde{\mathcal{F}}$, a measure $\mu$ with $\tilde{\Omega}$-marginal $\tilde{\mathbb{P}}$ is called \emph{$\tilde{\mathcal{G}}$-measurable} if there is a version $(\mu_\omega)$ of the disintegration of $\mu$ such that $\omega \mapsto \mu_\omega(A)$ is $\tilde{\mathcal{G}}$-measurable for all $A \in \mathcal{B}(X)$.
\\ \\
A \emph{random map (over $\Delta$)} on a measurable space $Y$ is a $\Delta$-indexed family $\mathbf{f}=(f_\alpha)_{\alpha \in \Delta}$ of functions $f_\alpha \colon Y \to Y$ such that $(\alpha,x) \mapsto f_\alpha(x)$ is measurable. We define \emph{$\mathbf{f}$-stationary measures} and \emph{$\mathbf{f}$-ergodic measures} as in Remark~\ref{ergmin}. If $Y$ is a standard Borel space and $\mathbf{f}$ admits a stationary measure, then $\mathbf{f}$ must admit an ergodic measure.
\\ \\
Now let $\mathbf{f}=(f_\alpha)_{\alpha \in \Delta}$ be a random circle homeomorphism. We define the \emph{inverse of $\mathbf{f}$} to be the random circle homeomorphism $\mathbf{f}^{-1}:=(f_\alpha^{-1})_{\alpha \in \Delta}$. For any $n \in \mathbb{N}$, over the probability space $(\Delta^n,\mathcal{B}(\Delta^n),\nu^{\otimes n})$ we define the random map $\mathbf{f}^n\!=\!(f_{\boldsymbol{\alpha}}^n)_{\boldsymbol{\alpha} \in \Delta^n} := (f_{\alpha_{n-1}} \circ \ldots \circ f_{\alpha_0})_{(\alpha_0,\ldots,\alpha_{n-1}) \in \Delta^n}$. (We define $\mathbf{f}^n$-invariant and $\mathbf{f}^n$-minimal sets just as in Definition~\ref{minimal}.)

\begin{prop} \label{n minimal}
If $\mathbb{S}^1$ is $\mathbf{f}$-minimal, then $\mathbb{S}^1$ is $\mathbf{f}^n$-minimal for all $n \in \mathbb{N}$.\footnote{This statement generalises to any random map $\mathbf{f}$ on a connected compact metric space such that $(\alpha,x) \mapsto f_\alpha(x)$ is continuous. (Connectedness of $\Delta$ is not needed.)}
\end{prop}

\begin{proof}
For a contradiction, suppose that $\mathbb{S}^1$ is $\mathbf{f}$-minimal and let $n \geq 2$ be the smallest integer such that $\mathbb{S}^1$ is not $\mathbf{f}^n$-minimal. Let $G_n$ be an $\mathbf{f}^n$-minimal set, and define the sets $G_0,\ldots,G_{n-1}$ by
\[ G_{n-k} \ := \ \bigcap_{\boldsymbol{\alpha} \in \Delta^k} (f_{\boldsymbol{\alpha}}^k)^{-1}(G_n) \ = \ \bigcap_{\alpha \in \Delta} f_\alpha^{-1}(G_{n-k+1}). \]
Since $G_n$ is $\mathbf{f}^n$-invariant, we have that $G_n \subset G_0$; hence
\[ f_\alpha\left( \bigcup_{r=1}^n G_r \right) \ \subset \ f_\alpha\left(\bigcup_{r=0}^{n-1} G_r \right) \ \subset \ \bigcup_{r=1}^n G_r \]
for all $\alpha \in \Delta$, and so $\bigcup_{r=1}^n G_r$ is $\mathbf{f}$-invariant. Moreover, for any $0 \leq r \leq n-1$, for any $\alpha_0,\ldots,\alpha_{n-1} \in \Delta$, we have that
\[ f_{(\alpha_0,\ldots,\alpha_{n-1})}^n(G_r) \ \subset \ f_{(\alpha_{n-r},\ldots,\alpha_{n-1})}^r(G_n) \ \subset \ f_{(\alpha_{n-r},\ldots,\alpha_{n-1})}^r(G_0) \ \subset \ G_r \]
and so $G_r$ is $\mathbf{f}^n$-invariant; but for each $1 \leq r \leq n-1$, $G_n$ is not $\mathbf{f}^{n-r}$-invariant, and so $G_n \not\subset G_r$. Hence, since $G_n$ is $\mathbf{f}^n$-minimal, we have that $G_r \cap G_n=\emptyset$ for each $1 \leq r \leq n-1$. So since $\mathbb{S}^1$ is connected, the $\mathbf{f}$-invariant set $\bigcup_{r=1}^n G_r$ is a proper subset of $\mathbb{S}^1$, contradicting that $\mathbb{S}^1$ is $\mathbf{f}$-minimal.
\end{proof}

\noindent We define the notations
\begin{align*}
\varphi_\mathbf{f}(n,\omega) \ :=& \ f_{(\alpha_0,\ldots,\alpha_{n-1})}^n \ = \ f_{\alpha_{n-1}} \circ \ldots \circ f_{\alpha_0} \\
\varphi_\mathbf{f}(-n,\omega) \ :=& \ (f_{(\alpha_{-n},\ldots,\alpha_{-1})}^n)^{-1} \ = \ (f_{\alpha_{-1}} \circ \ldots \circ f_{\alpha_{-n}})^{-1}
\end{align*}
for any $n \in \mathbb{N}_0$ and $\omega \in \Omega$. We also define $\Theta_\mathbf{f} \colon \Omega \times \mathbb{S}^1 \to \Omega \times \mathbb{S}^1$ by
\[ \Theta_\mathbf{f}(\omega,x) \ = \ (\theta\omega,f_{\alpha_0}(x)). \]
A \emph{$\varphi_\mathbf{f}$-invariant measure} (resp.~\emph{$\varphi_\mathbf{f}$-ergodic measure}) is a measure $\mu$ on $\Omega \times \mathbb{S}^1$ that is $\Theta_\mathbf{f}$-invariant (resp.~$\Theta_\mathbf{f}$-ergodic) and has $\Omega$-marginal $\mathbb{P}$. One can show that a measure $\mu$ with $\Omega$-marginal $\mathbb{P}$ is $\varphi_\mathbf{f}$-invariant if and only if $\mu_{\theta\omega} = f_{\alpha_0\ast}\mu_\omega$ for $\mathbb{P}$-almost all $\omega$. The $\varphi_\mathbf{f}$-ergodic measures are precisely the extreme points of the convex set of all $\varphi_\mathbf{f}$-invariant measures. A \emph{random fixed point of $\varphi_\mathbf{f}$} is a measurable function $a \colon \Omega \to \mathbb{S}^1$ such that $f_{\alpha_0}(a(\omega))=a(\theta\omega)$ for $\mathbb{P}$-almost all $\omega$. For any random fixed point $a$ of $\varphi_\mathbf{f}$, $(\delta_{a(\omega)})_{\omega \in \Omega}$ is a $\varphi_\mathbf{f}$-ergodic measure.
\\ \\
Define the sub-$\sigma$-algebras $\mathcal{F}_-$ and $\mathcal{F}_+$ of $\mathcal{F}$ by
\[ \mathcal{F}_- \, = \, \sigma(\omega \mapsto \alpha_i : i < 0) \hspace{4mm} \textrm{and} \hspace{4mm} \mathcal{F}_+ \, = \, \sigma(\omega \mapsto \alpha_i : i \geq 0). \]
It is well-known (\cite[Section~1.7]{Arn98}) that the map sending a measure $\mu$ on $\Omega \times \mathbb{S}^1$ to its  $\mathbb{S}^1$-marginal $\rho := \mu(\Omega \times \,\cdot\,)\,$ serves as a bijection between:
\begin{enumerate}[\indent (a)]
\item the set of $\mathcal{F}_-$-measurable $\varphi_\mathbf{f}$-invariant measures and the set of $\mathbf{f}$-stationary measures;
\item the set of $\mathcal{F}_-$-measurable $\varphi_\mathbf{f}$-ergodic measures and the set of $\mathbf{f}$-ergodic measures;
\item the set of $\mathcal{F}_+$-measurable $\varphi_\mathbf{f}$-invariant measures and the set of $\mathbf{f}^{-1}$-stationary measures;
\item the set of $\mathcal{F}_+$-measurable $\varphi_\mathbf{f}$-ergodic measures and the set of $\mathbf{f}^{-1}$-ergodic measures;
\end{enumerate}
\noindent with the inverse map $\rho \mapsto \mu$ being as follows:
\begin{itemize}
\item in case~(a)/(b), $\varphi_\mathbf{f}(n,\theta^{-n}\omega)_\ast\rho \,\to\, \mu_\omega$ weakly as $n \to \infty$ for $\mathbb{P}$-almost all $\omega$;
\item in case~(c)/(d), $\varphi_\mathbf{f}(-n,\theta^n\omega)_\ast\rho \,\to\, \mu_\omega$ weakly as $n \to \infty$ for $\mathbb{P}$-almost all $\omega$.
\end{itemize}

\noindent Now suppose we have two random circle homeomorphisms $\mathbf{f}$ and $\mathbf{g}$ that have topologically conjugate dynamics; an $\Omega$-indexed family $(h_\omega)$ of homeomorphisms $h_\omega \colon \mathbb{S}^1 \to \mathbb{S}^1$ fulfilling the description in Section~\ref{Main result} is called a \emph{topological conjugacy from $\varphi_\mathbf{f}$ to $\varphi_\mathbf{g}$}. In the case that $h_\omega$ is orientation-preserving for $\mathbb{P}$-almost all $\omega$, we refer to $(h_\omega)$ as an \emph{orientation-preserving conjugacy from $\varphi_\mathbf{f}$ to $\varphi_\mathbf{g}$}. Given a topological conjugacy $(h_\omega)$ from $\varphi_\mathbf{f}$ to $\varphi_\mathbf{g}$, the map $(\mu_\omega) \mapsto (h_{\omega\ast}\mu_\omega)$ serves as a bijection from the set of $\varphi_\mathbf{f}$-invariant (resp.\ $\varphi_\mathbf{f}$-ergodic) measures to the set of $\varphi_\mathbf{g}$-invariant (resp.\ $\varphi_\mathbf{g}$-ergodic) measures.

\subsection{Proof of Proposition~\ref{permute} and Remark~\ref{permrem}} \label{permproof}

Let $\mathbf{f}$ be a random circle homeomorphism. Since $\Delta$ is connected, any non-empty finite $\mathbf{f}$-invariant set $P$ can be enumerated $\{x_0,\ldots,x_{n-1}\}$ in such a manner that $f_\alpha(x_i) = x_{i+1 \;\mathrm{mod}\;n}$ for all $\alpha$ and $i$, contradicting our non-degeneracy assumption; so there is no non-empty finite $\mathbf{f}$-invariant set. Hence, by \cite[Theorem~B]{Mal14}, there are finitely many $\mathbf{f}$-minimal sets. We now aim to prove that each $\mathbf{f}$-minimal set has finitely many connected components; the rest of Proposition~\ref{permute}, together with all but the final statement in Remark~\ref{permrem}, then follows immediately from the fact that $\Delta$ is connected and $f_\alpha$ is orientation-preserving.
\\ \\
Let $G$ be the union of the $\mathbf{f}$-minimal sets. For any connected component $C$ of $G$ and any $\alpha \in \Delta$, since $f_\alpha(G) \subset G$ and $f_\alpha(C)$ is connected, we have that $f_\alpha(C)$ is contained in a connected component $\mathcal{Z}(C)$ of $G$; and since $\Delta$ is connected, $\mathcal{Z}(C)$ is independent of $\alpha$.
\\ \\
Now suppose for a contradiction that there is an $\mathbf{f}$-minimal set $K$ with infinitely many connected components, and fix any connected component $C$ of $K$. Since $K$ is $\mathbf{f}$-minimal, the sets $C, \ \mathcal{Z}(C), \ \mathcal{Z}^2(C), \, \ldots$ must all be distinct; so let $\{q_0\}$ be the limit of a convergent (in the Hausdorff metric) subsequence $(\mathcal{Z}^{m_n}(C))_{n \geq 0}$. Given any $\alpha \in \Delta$ and a sequence $(x_n)_{n \geq 0}$ with $x_n \in \mathcal{Z}^{m_n}(C)$ for each $n$, we obviously have that $f_\alpha(x_n) \to f_\alpha(q_0)$; but $f_\alpha(x_n) \in \mathcal{Z}^{m_n+1}(C)$ for each $n$, and so since the sets $\mathcal{Z}^{m_n+1}(C)$ are all distinct, we must have that $\mathcal{Z}^{m_n+1}(C)$ converges to the singleton $\{f_\alpha(q_0)\}$ as $n\to\infty$. So $q_1:=f_\alpha(q_0)$ is independent of $\alpha$. By repeating this argument we can obtain a sequence $(q_r)_{r \geq 0}$ in $K$ such that $f_\alpha(q_r)=q_{r+1}$ for all $\alpha \in \Delta$ and $r \geq 0$. Now given any $\alpha \in \Delta$ and a sequence $(y_n)_{n \geq 1}$ with $y_n \in \mathcal{Z}^{m_n-1}(C)$ for each $n$, we have that $y_n \to f_\alpha^{-1}(q_0)$, and so once again, $\mathcal{Z}^{m_n-1}(C)$ converges to the singleton $\{f_\alpha^{-1}(q_0)\}=:\{q_{-1}\}$. Repeating this argument, we can obtain $(q_r)_{r \leq 0}$ in $K$ such that $f_\alpha(q_{r-1})=q_r$ for all $\alpha \in \Delta$ and $r \leq 0$. Now $\{q_r:r \in \mathbb{Z}\} \subset K \neq \mathbb{S}^1$, so $\{q_r:r \in \mathbb{Z}\}$ is not dense, contradicting our non-degeneracy assumption. Thus we have proved that each $\mathbf{f}$-minimal set must have finitely many connected components.
\\ \\
We now prove the last statement in Remark~\ref{permrem}. Suppose that $\mathbf{f}$ has a unique minimal set, and this minimal set is an arc $G$. Take any $\tau \in \mathrm{Homeo}^+(\mathbb{S}^1)$ that commutes with $f_\alpha$ for every $\alpha$. Then $\tau(G)$ is $\mathbf{f}$-invariant, and so $G \subset \tau(G)$; since $G$ is an arc, it follows that $\tau$ has a fixed point in $G$, so in particular, $\tau$ is not a symmetry.

\subsection{Invariant measures of non-minimal random circle homeomorphisms}

\noindent Throughout the rest of this paper, we will often drop ``$\mathrm{mod} \ k_\mathbf{f}$'' from within subscripts and superscripts, when it is clear from the context; so for instance, $l_\mathbf{f}$ in Proposition~\ref{permute} is defined such that for all $\alpha$,
\[ f_\alpha(G_i) \ \subset \ G_{i+l_\mathbf{f}}. \]

\begin{lemma} \label{Li}
Let $\mathbf{f}$ be a random circle homeomorphism such that $\mathbb{S}^1$ is not $\mathbf{f}$-minimal. Let $G$, $k_\mathbf{f}$ and $l_\mathbf{f}$ be as in Proposition~\ref{permute}, with $G_0,\ldots,G_{k_\mathbf{f}-1}$ being the connected components of $G$ going anticlockwise round the circle, and let $p$ and $q$ be as in Remark~\ref{permrem}. For each $0 \leq i \leq k_\mathbf{f}-1$, let $H_i:=[\partial_+G_i,\partial_-G_{i+1}]$, and let $L_i:=\bigcup_{j=0}^q H_{i+jp}$. Then $L_i$ is $\mathbf{f}^{-1}$-invariant, and there is exactly one $\mathbf{f}^{-1}$-stationary measure assigning full measure to $L_i$.
\end{lemma}

\begin{proof}
Fix $i$. We have that $H_{i+jp+l_\mathbf{f}} \subset f_\alpha(H_{i+jp})$ for all $\alpha$ and $j$, so $L_i$ is $\mathbf{f}^{-1}$-invariant. Suppose for a contradiction that there are two distinct $\mathbf{f}^{-1}$-ergodic measures $\rho_1$ and $\rho_2$ assigning full measure to $L_i$, with supports $K_1$ and $K_2$. Let $G'$ be the union of the $\mathbf{f}^{-1}$-minimal sets, and let $U$ be the union of all those connected components of $\mathbb{S}^1 \setminus G'$ that are contained in $L_i$; note that $K_1$ and $K_2$ intersect $H_{i+jp}$ for all $j$, and so $U$ is non-empty. Moreover, since $G \setminus G'$ is $\mathbf{f}$-invariant and the boundary points of $U$ belong to $\mathbf{f}^{-1}$-minimal sets, we must have that $\bar{U}$ is $\mathbf{f}$-invariant. Hence $\bar{U}$ must contain some connected component of $G$, contradicting that $U \subset L_i$.
\end{proof}

\noindent For the next lemma, we introduce the following notations (representing ``partially-strict monotone convergence''):
\begin{itemize}
\item we write ``$x_n \searrow x$'' to mean ``$x_n \to x$ and $x_{n+1} \in \;]x,x_n]$ for all $n \geq 0$'';
\item we write ``$x_n \nearrow x$'' to mean ``$x_n \to x$ and $x_{n+1} \in [x_n,x[$ for all $n \geq 0$''.
\end{itemize}

\begin{lemma} \label{pull}
In the setting of Lemma~\ref{Li}, there exist $\mathcal{F}_-$-measurable functions $a_{i\,} \colon \Omega \to G_i^\circ$ and $\mathcal{F}_+$-measurable functions $r_{i\,} \colon \Omega \to H_i^\circ$ (with $0 \leq i \leq k-1)$, such that: (i)~for each $i$, the measures
\[ \left(\frac{1}{q}\,\sum_{j=0}^{q-1} \delta_{a_{i+jp}(\omega)}\right)_{\omega \in \Omega} \hspace{2mm}\textrm{and}\hspace{4mm} \left(\frac{1}{q}\,\sum_{j=0}^{q-1} \delta_{r_{i+jp}(\omega)}\right)_{\omega \in \Omega} \]
are $\varphi_\mathbf{f}$-invariant; and (ii)~for $\mathbb{P}$-almost all $\omega$, for each $0 \leq i \leq k-1$,
\[ \hspace{-1mm} \begin{array}{l l l}
u_n^i(\omega)\,:=\,\varphi_\mathbf{f}(n,\theta^{-n}\omega)(\partial_-G_{i-nl_\mathbf{f}}) \nearrow a_i(\omega), & & u_{-n}^i(\omega)\,:=\,\varphi_\mathbf{f}(-n,\theta^n\omega)(\partial_-G_{i+nl_\mathbf{f}}) \searrow r_{i-1}(\omega), \\
\hspace{0.3mm}v_n^i(\omega)\,:=\,\varphi_\mathbf{f}(n,\theta^{-n}\omega)(\partial_+G_{i-nl_\mathbf{f}}) \searrow a_i(\omega), & & \hspace{0.3mm}v_{-n}^i(\omega)\,:=\,\varphi_\mathbf{f}(-n,\theta^n\omega)(\partial_+G_{i+nl_\mathbf{f}}) \nearrow r_i(\omega).
\end{array}\]
\end{lemma}

\begin{proof}
For every $\omega$, we have that $u_n^i(\omega) \in G_i$ and $u_{n+1}^i(\omega) \in [u_n^i(\omega),\partial_+G_i]$ for all $n \geq 0$, and so $u_n^i(\omega)$ converges to some value $a_i^-(\omega) \in G_i$. Moreover, since no non-empty $\mathbf{f}$-invariant set is finite, we have that for $\mathbb{P}$-almost all $\omega$ there is an unbounded sequence $(m_n)$ in $\mathbb{N}$ such that $f_{\alpha_{-(m_n+1)}}(\partial_-G_{i-(m_n+1)l_\mathbf{f}}) \, \neq \, \partial_-G_{i-m_nl_\mathbf{f}}$ and so $u_{m_n+1}^i(\omega) \, \neq \, u_{m_n}^i(\omega)$. Thus $u_n^i(\omega) \nearrow a_i^-(\omega)$ (so in particular $a_i^-(\omega) \neq \partial_-G_i$) for $\mathbb{P}$-almost all $\omega$. Moreover, $f_{\alpha_0}(a_i^-(\omega))=a_{i+l_\mathbf{f}}^-(\theta\omega)$ for every $\omega$, and therefore the measure $\mu^{i,a-}$ with disintegration $\left(\frac{1}{q}\,\sum_{j=0}^{q-1} \delta_{a_{i+jp}^-(\omega)}\right)$ is $\varphi_\mathbf{f}$-invariant.
\\ \\
One can similarly find random variables
\begin{align*}
a_i^+ &\colon \Omega \to G_i \setminus \{\partial_+G_i\} \\
r_i^- &\colon \Omega \to H_i \setminus \{\partial_-H_i\} \\
r_i^+ &\colon \Omega \to H_i \setminus \{\partial_+H_i\}
\end{align*}
such that $v_n^i(\omega) \searrow a_i^+(\omega)$, $\,v_{-n}^i(\omega) \nearrow r_i^-(\omega)$, and $u_{-n}^i(\omega) \searrow r_{i-1}^+(\omega)$, and moreover the measures $\mu^{i,a+}$, $\mu^{i,r-}$ and $\mu^{i,r+}$ with disintegrations $\left(\frac{1}{q}\,\sum_{j=0}^{q-1} \delta_{a_{i+jp}^+(\omega)}\right)$, $\left(\frac{1}{q}\,\sum_{j=0}^{q-1} \delta_{r_{i+jp}^-(\omega)}\right)$ and $\left(\frac{1}{q}\,\sum_{j=0}^{q-1} \delta_{r_{i+jp}^+(\omega)}\right)$ (respectively) are $\varphi_\mathbf{f}$-invariant.
\\ \\
Now $\mu^{i,a-}$ and $\mu^{i,a+}$ are both $\mathcal{F}_-$-measurable $\varphi_\mathbf{f}$-invariant measures assigning full probability to $\Omega \times M_i$. But since $M_i$ only supports one $\mathbf{f}$-stationary measure (Remark~\ref{ergmin}), we must have that $\mu^{i,a-}=\mu^{i,a+}$, and therefore $a_i^-=a_i^+$ $\mathbb{P}$-a.s..
\\ \\
Likewise, $\mu^{i,r-}$ and $\mu^{i,r+}$ are both $\mathcal{F}_+$-measurable $\varphi_\mathbf{f}$-invariant measures assigning full probability to $\Omega \times L_i$. But since, by Lemma~\ref{Li}, $L_i$ only supports one $\mathbf{f}^{-1}$-stationary measure, we must have that $\mu^{i,r-}=\mu^{i,r+}$, and therefore $r_i^-=r_i^+$ $\mathbb{P}$-a.s..
\end{proof}

\subsection{Invariant measures of minimal random circle homeomorphisms}

\begin{lemma} \label{min inv}
Let $\mathbf{f}$ be a random circle homeomorphism such that $\mathbb{S}^1$ is $\mathbf{f}$-minimal. Then $\mathbb{S}^1$ is also $\mathbf{f}^{-1}$-minimal.
\end{lemma}

\begin{proof}
If $\mathbb{S}^1$ is not $\mathbf{f}^{-1}$-minimal, then letting $G'$ be the union of the $\mathbf{f}^{-1}$-minimal sets, since no connected component $G'$ is a singleton, we have that $\overline{\mathbb{S}^1 \setminus G'}$ is an $\mathbf{f}$-invariant proper subset of $\mathbb{S}^1$.
\end{proof}

\begin{defi} \label{tm}
For any $m \in \mathbb{N}$, let $\tau_m \colon \mathbb{S}^1\to\mathbb{S}^1$ be the rotation $x \mapsto x+[\frac{1}{m}]$. Given a map $f \in \mathrm{Homeo}(\mathbb{S}^1)$ commuting with $\tau_m$, define $z_m(f) \in \mathrm{Homeo}(\mathbb{S}^1)$ by $z_m(f)(mx)=mf(x)$.
\end{defi}

\begin{rmk} \label{tmrem}
$z_m$ serves as a group homomorphism from the set of homeomorphisms commuting with $\tau_m$ to the set of all homeomorphisms of $\mathbb{S}^1$; if $m \in \{1,2\}$ then $z_m$ is surjective. However, if $m \geq 3$, then any homeomorphism commuting with $\tau_m$ must be orientation-preserving, and the image of $z_m$ is the set of all orientation-preserving homeomorphisms. For any $m$, given $f,g \in \mathrm{Homeo}(\mathbb{S}^1)$ commuting with $\tau_m$, we have that $z_m(f)=z_m(g)$ if and only if there exists $i \in \{0,\ldots,m-1\}$ such that $g = \tau_m^i \circ f$.
\end{rmk}

\noindent Now observe that if $\tau$ is a symmetry of a random circle homeomorphism $(f_\alpha)$, then for any homeomorphism $h \in \mathrm{Homeo}(\mathbb{S}^1)$, $h \circ \tau \circ h^{-1}$ is a symmetry of the random circle homeomorphism $(h \circ f_\alpha \circ h^{-1})$.

\begin{defi}
Let $\mathbf{f}=(f_\alpha)$ be a random circle homeomorphism admitting a symmetry $\tau$ of order $m$. A \emph{$\tau$-lift of $\mathbf{f}$} is a random map $\mathbf{F}$ on $\mathbb{S}^1$ of the form $\mathbf{F}=(z_m(h \circ f_\alpha \circ h^{-1}))_{\alpha \in \Delta}$ where $h \in \mathrm{Homeo}^+(\mathbb{S}^1)$ is such that $h \circ \tau \circ h^{-1}=\tau_m^i$ for some $i$.
\end{defi}

\begin{rmk}
Although the $\tau$-lift of $\mathbf{f}$ is not uniquely defined, any two $\tau$-lifts of $\mathbf{f}$ are deterministically orientationally conjugate. Also note that if $\mathbb{S}^1$ is $\mathbf{f}$-minimal then $\mathbb{S}^1$ is $\mathbf{F}$-minimal for any $\tau$-lift $\mathbf{F}$ of $\mathbf{f}$.
\end{rmk}

\begin{defi} \label{contractive}
We say that a random circle homeomorphism $\mathbf{f}$ is \emph{contractive} if $\varphi_\mathbf{f}$ has an $\mathcal{F}_-$-measurable random fixed point $a \colon \Omega \to \mathbb{S}^1$ and an $\mathcal{F}_+$-measurable random fixed point $r \colon \Omega \to \mathbb{S}^1$ such that for $\mathbb{P}$-almost all $\omega$, for every $x \in \mathbb{S}^1 \setminus \{r(\omega)\}$, $d(\varphi_\mathbf{f}(n,\omega)x,a(\theta^n\omega)) \to 0$ as $n \to \infty$. We refer to $a$ and $r$ respectively as the \emph{attractor} and the \emph{repeller} of $\mathbf{f}$.
\end{defi}

\noindent Note that the property of being contractive is preserved under deterministic topological conjugacy.

\begin{defi}
A \emph{random rotation} is a random circle homeomorphism $\mathbf{f}$ such that for every $\alpha \in \Delta$ there exists $s(\alpha) \in \mathbb{S}^1$ such that $f_\alpha(x)=x+s(\alpha)$ for all $x$.
\end{defi}

\begin{prop} \label{antonov}
Let $\mathbf{f}$ be a random circle homeomorphism, and suppose that $\mathbb{S}^1$ is $\mathbf{f}$-minimal. Then exactly one of the following statements holds:
\begin{enumerate}[\indent (a)]
\item $\mathbf{f}$ is contractive;
\item $\mathbf{f}$ admits a symmetry $\tau$ such that the $\tau$-lifts of $\mathbf{f}$ are contractive;
\item $\mathbf{f}$ is deterministically orientationally conjugate to a random rotation.
\end{enumerate}
\end{prop}

\noindent Proposition~\ref{antonov} is essentially the main result of \cite{Ant84} (with the condition of inverse-minimality being automatically fulfilled due to Lemma~\ref{min inv}), except that the notion of contractivity in \cite{Ant84} is not formulated in the same way that we do here; nonetheless, contractivity according to our formulation can be deduced using \cite[Theorems~2.10 and 5.13]{newman2}.
\\ \\
We now look at the $\varphi_\mathbf{f}$-invariant measures in each of the cases in Proposition~\ref{antonov}. For convenience, we define an \emph{$\iota$-symmetry} of a random circle homeomorphism $\mathbf{f}$ to be a map that is either the identity function on $\mathbb{S}^1$ or a symmetry of $\mathbf{f}$.

\begin{lemma} \label{sym contr}
Let $\mathbf{f}$ be a random circle homeomorphism such that $\mathbb{S}^1$ is $\mathbf{f}$-minimal. Suppose that for some $m \in \mathbb{N}$, $\tau_m$ is an $\iota$-symmetry of $\mathbf{f}$, and $\mathbf{F}:=(z_m(f_\alpha))$ is contractive, with attractor $A$ and repeller $R$. Let $a_1(\omega),\ldots,a_m(\omega) \in \mathbb{S}^1$ be the points whose $m$-th multiple is $A(\omega)$, and let $r_1(\omega),\ldots,r_m(\omega) \in \mathbb{S}^1$ be the points whose $m$-th multiple is $R(\omega)$.

(A) Suppose we have a random probability measure $(p_\omega)_{\omega \in \Omega}$ on $\mathbb{S}^1$ such that for $\mathbb{P}$-almost all $\omega$, $p_\omega(\{r_1(\omega),\ldots,r_m(\omega)\})=0$ and $\varphi_\mathbf{f}(n,\theta^{-n}\omega)_\ast p_{\theta^{-n}\omega}$ converges weakly as $n\to\infty$ to some measure $u_\omega$. Then $u_\omega = \frac{1}{m}(\delta_{a_1(\omega)} + \ldots + \delta_{a_m(\omega)})$ for $\mathbb{P}$-almost all $\omega$.

(B) Hence there are exactly two $\varphi_\mathbf{f}$-ergodic measures, namely $(\frac{1}{m}(\delta_{a_1(\omega)} + \ldots + \delta_{a_m(\omega)}))_{\omega \in \Omega}$ and $(\frac{1}{m}(\delta_{r_1(\omega)} + \ldots + \delta_{r_m(\omega)}))_{\omega \in \Omega}$.
\end{lemma}

\begin{proof}
(A)~For any measure $\mu$ on $\mathbb{S}^1$, define the measure $\bar{\mu}$ on $\mathbb{S}^1$ by $\bar{\mu}(B):=\mu(x \in \mathbb{S}^1 : mx \in B)$. Given any continuous function $g \colon \mathbb{S}^1 \to \mathbb{R}$, we have that for $\mathbb{P}$-almost all $\omega$, $\bar{p}_\omega(\{R(\omega)\})=0$ and so (by the dominated convergence theorem),
\[ \int_{\mathbb{S}^1} g(\varphi_\mathbf{F}(n,\omega)x) \, \bar{p}_\omega(dx) \ - \ g(A(\theta^n\omega)) \ \to \ 0 \, ; \]
so there is a sequence $m_n \to \infty$ such that
\[ \int_{\mathbb{S}^1} g(\varphi_\mathbf{F}(n,\theta^{-m_n}\omega)x) \, \bar{p}_{\theta^{-m_n}\omega}(dx) \ \to \ g(A(\omega)) \hspace{3mm} \textrm{for $\mathbb{P}$-a.e.\ $\omega$.} \]
But we know that for any continuous function $g \colon \mathbb{S}^1 \to \mathbb{R}$,
\[ \int_{\mathbb{S}^1} g(\varphi_\mathbf{F}(n,\theta^{-n}\omega)x) \, \bar{p}_{\theta^{-n}\omega}(dx) \ \to \ \int_{\mathbb{S}^1} g(x) \, \bar{u}_\omega(dx) \hspace{3mm} \textrm{for $\mathbb{P}$-a.e.\ $\omega$.} \]
So $\bar{u}_\omega=\delta_{A(\omega)}$ for $\mathbb{P}$-almost all $\omega$. Hence we are done if $m=1$. Now in the case that $m \geq 2$: It is clear that $(\frac{1}{m}(\delta_{a_1(\omega)} + \ldots + \delta_{a_m(\omega)}))_{\omega \in \Omega}$ is $\varphi_\mathbf{f}$-invariant; so since there is a unique $\mathbf{f}$-ergodic measure, $(\frac{1}{m}(\delta_{a_1(\omega)} + \ldots + \delta_{a_m(\omega)}))$ is $\varphi_\mathbf{f}$-ergodic. But $(u_\omega)$ is also clearly $\varphi_\mathbf{f}$-invariant, so we must have that $u_\omega = \frac{1}{m}(\delta_{a_1(\omega)} + \ldots + \delta_{a_m(\omega)})$ almost surely. (B)~Just as we have seen that $(\frac{1}{m}(\delta_{a_1(\omega)} + \ldots + \delta_{a_m(\omega)}))$ is $\varphi_\mathbf{f}$-ergodic, so also $(\frac{1}{m}(\delta_{r_1(\omega)} + \ldots + \delta_{r_m(\omega)}))$ is $\varphi_\mathbf{f}$-ergodic (using Lemma~\ref{min inv} in the case that $m \geq 2$). Any two distinct $\varphi_\mathbf{f}$-ergodic measures are mutually singular, and so for any $\varphi_\mathbf{f}$-ergodic measure $(p_\omega)$ that is distinct from $(\frac{1}{m}(\delta_{r_1(\omega)} + \ldots + \delta_{r_m(\omega)}))$, part~(A) gives that $p_\omega=u_\omega$ for $\mathbb{P}$-almost all $\omega$.
\end{proof}

\noindent Although we will mostly consider topological conjugacy in the next section, the following corollary of Lemma~\ref{sym contr} is worth mentioning now:

\begin{cor} \label{pres attr}
Let $\mathbf{f}$ and $\mathbf{g}$ be random circle homeomorphisms such that $\mathbb{S}^1$ is $\mathbf{f}$-minimal and $\mathbf{g}$-minimal. Suppose that for some $m \in \mathbb{N}$, $\tau_m$ is an $\iota$-symmetry of both $\mathbf{f}$ and $\mathbf{g}$, and $\mathbf{F}:=(z_m(f_\alpha))$ is contractive with attractor $A_\mathbf{f}$ and repeller $R_\mathbf{f}$, and $\mathbf{G}:=(z_m(g_\alpha))$ is contractive with attractor $A_\mathbf{g}$ and repeller $R_\mathbf{g}$. Let $\mathcal{A}_\mathbf{f}(\omega)$ (resp.\ $\mathcal{A}_\mathbf{g}(\omega)$, $\mathcal{R}_\mathbf{f}(\omega)$, $\mathcal{R}_\mathbf{g}(\omega)$) be the set of points whose $m$-th multiple is $A_\mathbf{f}(\omega)$ (resp.\ $A_\mathbf{g}(\omega)$, $R_\mathbf{f}(\omega)$, $R_\mathbf{g}(\omega)$). Suppose we have a topological conjugacy $(\tilde{h}_\omega)$ from $\varphi_\mathbf{f}$ to $\varphi_\mathbf{g}$. Then for $\mathbb{P}$-almost all $\omega$,
\[ \tilde{h}_\omega(\mathcal{A}_\mathbf{f}(\omega)) \ = \ \mathcal{A}_\mathbf{g}(\omega) \hspace{4mm} \textrm{and} \; \hspace{4mm} \tilde{h}_\omega(\mathcal{R}_\mathbf{f}(\omega)) \ = \ \mathcal{R}_\mathbf{g}(\omega). \]
\end{cor}

\begin{proof}
Given a finite set $P \subset \mathbb{S}^1$, write $\lambda_P$ for the probability measure supported uniformly on $P$. Let $\rho_\mathbf{f}$ be the unique $\mathbf{f}$-stationary measure, and let $p_\omega:=\tilde{h}_{\omega\ast}\rho_\mathbf{f}$. For $\mathbb{P}$-almost every $\omega$, we have that
\[ \varphi_\mathbf{g}(n,\theta^{-n}\omega)_\ast p_{\theta^{-n}\omega} \, = \, \tilde{h}_{\omega\ast}(\varphi_\mathbf{f}(n,\theta^{-n}\omega)_\ast\rho_\mathbf{f}) \ \, \to \ \, \tilde{h}_{\omega\ast}\lambda_{\mathcal{A}_\mathbf{f}(\omega)}. \]
\noindent Now $p_\omega(\mathcal{R}_\mathbf{f}(\omega))=0$ for all $\omega$, since $\rho_\mathbf{f}$ is atomless; so then, applying Lemma~\ref{sym contr}(A) to $\mathbf{g}$ gives that $\tilde{h}_{\omega\ast}\lambda_{\mathcal{A}_\mathbf{f}(\omega)}=\lambda_{\mathcal{A}_\mathbf{g}(\omega)}$ for $\mathbb{P}$-almost all $\omega$. Since topological conjugacy preserves ergodic measures, due to Lemma~\ref{sym contr}(B) we also have that $\tilde{h}_{\omega\ast}\lambda_{\mathcal{R}_\mathbf{f}(\omega)}=\lambda_{\mathcal{R}_\mathbf{g}(\omega)}$ for $\mathbb{P}$-almost all $\omega$.
\end{proof}

\noindent In the following lemma, we use the ``partially-strict monotone convergence'' notations $x_n \searrow x$ and $x_n \nearrow x$ introduced immediately before Lemma~\ref{pull}.

\begin{lemma} \label{pull2}
Let $\mathbf{f}$ be a random circle homeomorphism such that $\mathbb{S}^1$ is $\mathbf{f}$-minimal, and suppose that $\mathbf{f}$ is contractive, with attractor $a_0$ and repeller $r_0$. There exist measurable functions $u,v \colon \Omega \to \mathbb{S}^1$ with $u(\omega) \in \;]r_0(\omega),a_0(\omega)[$ and $v(\omega) \in \;]a_0(\omega),r_0(\omega)[$ almost surely, such that for $\mathbb{P}$-almost every $\omega$,
\[ \hspace{-1mm} \begin{array}{l l l}
u_n(\omega)\,:=\,\varphi_\mathbf{f}(n,\theta^{-n}\omega)u(\theta^{-n}\omega) \nearrow a_0(\omega), & & u_{-n}(\omega)\,:=\,\varphi_\mathbf{f}(-n,\theta^n\omega)u(\theta^n\omega) \searrow r_0(\omega), \\
\hspace{0.3mm}v_n(\omega)\,:=\,\varphi_\mathbf{f}(n,\theta^{-n}\omega)v(\theta^{-n}\omega) \searrow a_0(\omega), & & \hspace{0.3mm}v_{-n}(\omega)\,:=\,\varphi_\mathbf{f}(-n,\theta^n\omega)v(\theta^n\omega) \nearrow r_0(\omega).
\end{array}\]
\end{lemma}

\begin{proof}
We will prove the existence of $u$; the existence of $v$ is proved similarly. Let $0<\varepsilon<1$ be such that $\mathbb{P}(\omega: d_+(r_0(\omega),a_0(\omega))>\varepsilon)>0$. For $\mathbb{P}$-almost every $\omega$, let $u(\omega) \in \mathbb{S}^1$ be such that
\[ \bigcap_{n=0}^\infty \varphi_\mathbf{f}(-n,\theta^n\omega)[a_0(\theta^n\omega)-[\varepsilon],a_0(\theta^n\omega)] \ = \ [u(\omega),a_0(\omega)]. \]
For $\mathbb{P}$-almost all $\omega$, there exists $n \geq 0$ such that $d_+(r_0(\theta^n\omega),a_0(\theta^n\omega))>\varepsilon$, and so $u(\omega) \in \;]r_0(\omega),a_0(\omega)]$. Note that
\[ \varphi_\mathbf{f}(1,\omega)[u(\omega),a_0(\omega)] \ \subset \ [u(\theta\omega),a_0(\theta\omega)] \, ; \]
so for any $n \geq 0$, $\varphi_\mathbf{f}(1,\theta^{-(n+1)}\omega)u(\theta^{-(n+1)}\omega) \in [u(\theta^{-n}\omega),a_0(\theta^{-n}\omega)]$ and so $u_{n+1}(\omega) \in [u_n(\omega),a_0(\omega)]$. So then, as $n \to \infty$, $u_n(\omega)$ converges to a value $b(\omega) \in \;]r_0(\omega),a_0(\omega)]$; but then $b(\cdot)$ is a random fixed point distinct from $r_0$, and so (by Lemma~\ref{sym contr}(B)), $b=a_0$ $\mathbb{P}$-a.s.. Now for $\mathbb{P}$-almost all $\omega$, letting $c_\omega$ be a point in $]r_0(\omega),a_0(\omega)[$, we have that $d_+(\varphi_\mathbf{f}(n,\omega)c_\omega,a_0(\theta^n\omega)) \leq \varepsilon$ for sufficiently large $n$, and therefore $u(\theta^n\omega) \neq a_0(\theta^n\omega)$ for sufficiently large $n$. Since $\mathbb{P}$ is $\theta$-invariant, it follows that $u(\omega) \neq a_0(\omega)$ for $\mathbb{P}$-almost all $\omega$, and therefore $u_n(\omega) \neq a(\omega)$ for $\mathbb{P}$-almost all $\omega$.
\\ \\
Thus, overall, we have that $u_n \nearrow a_0$ $\mathbb{P}$-a.s.. Now since $\varphi_\mathbf{f}(1,\omega)u(\omega) \in [u(\theta\omega),a_0(\theta\omega)]$ and $u(\theta\omega) \in \;]r_0(\theta\omega),a_0(\theta\omega)[$ for $\mathbb{P}$-almost all $\omega$, it follows that $u(\theta\omega) \in \;]r_0(\theta\omega),\varphi_\mathbf{f}(1,\omega)u(\omega)]$ and therefore $u_{-1}(\omega) \in \;]r_0(\omega),u(\omega)]$ for $\mathbb{P}$-almost all $\omega$. As above, it follows that $u_{-(n+1)}(\omega) \in \;]r_0(\omega),u_{-n}(\omega)]$ for all $n \geq 0$, $\mathbb{P}$-almost surely; so, since $u(\omega) \in \;]r_0(\omega),a_0(\omega)[$, we have (again using Lemma~\ref{sym contr}(B)) that $u_{-n} \searrow r_0$ $\mathbb{P}$-a.s..
\end{proof}

\noindent It remains to consider invariant measures for random rotations. We start with the following known result:

\begin{prop} \label{quas}
Let $(X,\circ)$ be a compact metrisable abelian topological group, with $\lambda$ the Haar probability measure. Let $r \colon \Omega \to X$ be a measurable function, and define $\Theta \colon \Omega \times X \to \Omega \times X$ by $\Theta(\omega,x)=(\theta\omega,x+r(\omega))$. Suppose that $\mathbb{P} \otimes \lambda$ is ergodic with respect to $\Theta$. Then $\mathbb{P} \otimes \lambda$ is the \emph{only} $\Theta$-invariant measure with $\Omega$-marginal $\mathbb{P}$.
\end{prop}

\noindent The above result does not rely on any structure of the dynamical system $(\Omega,\mathcal{F},\mathbb{P},\theta)$ other than that $\mathbb{P}$ is an ergodic measure of $\theta$. A simple proof due to Anthony Quas can be found at \cite{211302}.\footnote{Other proofs such as in \cite[Proposition~3.10]{MR603625} are more complicated, and rely on $(\Omega,\mathcal{F})$ being a standard Borel space.} With Proposition~\ref{quas}, we immediately have the following:

\begin{lemma} \label{rot inv}
Let $\mathbf{f}$ be a random circle homeomorphism such that $\mathbb{S}^1$ is $\mathbf{f}$-minimal, and suppose that $\mathbf{f}$ is deterministically orientationally conjugate to a random rotation. Let $\rho_\mathbf{f}$ be the unique $\mathbf{f}$-stationary measure. Then $\mathbb{P} \otimes \rho_\mathbf{f}$ is the only $\varphi_\mathbf{f}$-invariant measure.
\end{lemma}

\begin{cor} \label{rot rfp}
Let $s \colon \Delta \to \mathbb{S}^1$ be a continuous function not identically equal to $[0]$. There does not exist a measurable function $\kappa \colon \Omega \to \mathbb{S}^1$ with the property that for $\mathbb{P}$-almost all $\omega$, $\kappa(\theta\omega)=\kappa(\omega)+s(\alpha_0)$.
\end{cor}

\begin{proof}
If $s$ is identically equal to some rational $[\frac{p}{q}]$, then $\kappa(\theta^q\omega)=\kappa(\omega)$ for $\mathbb{P}$-almost all $\omega$, so (since $\mathbb{P}$ is $\theta^q$-ergodic) $\kappa$ is almost everywhere constant, contradicting that $\kappa(\theta\omega)=\kappa(\omega)+s(\alpha_0)$ for $\mathbb{P}$-almost all $\omega$. If $s$ is not identically equal to a rational point, then $\mathbb{S}^1$ is minimal under the random map $(x \mapsto x+s(\alpha))_{\alpha \in \Delta}$, and so the result follows from Lemma~\ref{rot inv}.
\end{proof}

\subsection{Perturbation of two-sided sequences}

In Lemmas~\ref{pull} and \ref{pull2}, we have constructed two-sided stochastic processes exhibiting ``monotone convergence''. However, for the proof of the main theorem, this monotonicity needs to be strengthened to \emph{strict} monotonicity. In this section, we will introduce a trick that achieves this.

\begin{rmk}
In the case that $\mathbf{f}$ has atomless transition probabilities (meaning that $\nu(\alpha:f_\alpha(x)=y)=0$ for all $x,y \in \mathbb{S}^1$), the convergences in Lemma~\ref{pull} are already strictly monotone. However, one can show that for $u$ as constructed in the proof of Lemma~\ref{pull2}, the convergences $u_n \nearrow a_0$ and $u_{-n} \searrow r_0$ cannot be strictly monotone.
\end{rmk}

\begin{defi}
Let $(f_n)$ be a two-sided sequence in $\mathrm{Homeo}^+(\mathbb{S}^1)$. An \emph{orbit of $(f_n)$} is a two-sided sequence $(a_n)$ in $\mathbb{S}^1$ such that $f_n(a_n)=a_{n+1}$ for all $n \in \mathbb{Z}$.
\end{defi}

\begin{defi} \label{pert}
Let $(f_n)_{n \in \mathbb{Z}}$ be a two-sided sequence in $\mathrm{Homeo}^+(\mathbb{S}^1)$, and let $(x_n)_{n \in \mathbb{Z}}$ be a two-sided sequence in $\mathbb{S}^1$. We say that \emph{$(x_n)$ is $(f_n)$-generic} if the set $P:=\{n \in \mathbb{Z} : f_n(x_n) \neq x_{n+1}\}$ is unbounded both above and below. In this case, we define the \emph{anticlockwise perturbation} (resp.\ \emph{clockwise perturbation}) \emph{of $(x_n)$ by $(f_n)$} to be the two-sided sequence $(y_n)$ such that:
\begin{itemize}
\item for all $n \in P$, $y_n=x_n$ for all $n \in P$;
\item for all $n \nin P$, $y_n$ is the midpoint of the arc $[x_n,f_{n-1}(y_{n-1})]$ (resp.\ of the arc $[f_{n-1}(y_{n-1}),x_n]$).
\end{itemize}
\end{defi}

\begin{lemma} \label{monotone pert}
Let $(f_n)$ be a two-sided sequence in $\mathrm{Homeo}^+(\mathbb{S}^1)$, and let $(a_n)$ be an orbit of $(f_n)$.
\begin{enumerate}[\indent (I)]
\item Let $(x_n)$ be an $(f_n)$-generic two-sided sequence such that $f_{n-1}(x_{n-1}) \in [x_n,a_n[\,$ for all $n$. Let $(y_n)$ be the anticlockwise perturbation of $(x_n)$ by $(f_n)$. Then for all $n$, we have that $f_{n-1}(y_{n-1}) \in \;]y_n,a_n[$. Hence, writing
\[ z_n \ := \ \left\{ \begin{array}{c l} f_{-1} \circ \ldots \circ f_{-n}(y_{-n}) & n > 0 \\ (f_{-n-1} \circ \ldots \circ f_0)^{-1}(y_{-n}) & n < 0 \\ y_0 & n=0, \end{array} \right. \]
we have that $z_n \in \;]z_{n-1},a_0[$ for all $n \in \mathbb{Z}$.
\item Let $(x_n)$ be an $(f_n)$-generic two-sided sequence such that $f_{n-1}(x_{n-1}) \in \;]a_n,x_n]\,$ for all $n$. Let $(y_n)$ be the clockwise perturbation of $(x_n)$ by $(f_n)$. Then for all $n$, we have that $f_{n-1}(y_{n-1}) \in \;]a_n,y_n[$. Hence, defining $(z_n)$ with reference to $(y_n)$ exactly as in part~(I), we have that $z_n \in \;]a_0,z_{n-1}[$ for all $n \in \mathbb{Z}$.
\end{enumerate}
\end{lemma}

\noindent The proof is a straightforward exercise (and visually quite clear).
\\ \\
In the following corollary, we write $x_n \sesearrows x$ (resp.\ $x_n \nenearrows x$) to mean: ``$x_n \searrow x$ (resp.\ $x_n \nearrow x$) with $x_n \neq x_{n+1}$ for all $n \geq 0$''.

\begin{cor} \label{strict pull}
(A) In the setting of Lemma~\ref{pull2}, there exist measurable functions $\tilde{u},\tilde{v} \colon \Omega \to \mathbb{S}^1$ such that for $\mathbb{P}$-almost every $\omega$,
\[ \hspace{-1mm} \begin{array}{l l l}
\tilde{u}_n^0(\omega)\,:=\,\varphi_\mathbf{f}(n,\theta^{-n}\omega)(\tilde{u}(\theta^{-n}\omega)) \nenearrows a_0(\omega), & & \tilde{u}_{-n}^0(\omega)\,:=\,\varphi_\mathbf{f}(-n,\theta^n\omega)(\tilde{u}(\theta^n\omega)) \sesearrows r_0(\omega), \\
\hspace{0.3mm}\tilde{v}_n^0(\omega)\,:=\,\varphi_\mathbf{f}(n,\theta^{-n}\omega)(\tilde{v}(\theta^{-n}\omega)) \sesearrows a_0(\omega), & & \hspace{0.3mm}\tilde{v}_{-n}^0(\omega)\,:=\,\varphi_\mathbf{f}(-n,\theta^n\omega)(\tilde{v}(\theta^n\omega)) \nenearrows r_0(\omega).
\end{array}\]

(B) In the setting of Lemma~\ref{pull}, there exist measurable functions $\tilde{u}^i,\tilde{v}^i \colon \Omega \to \mathbb{S}^1$, $0 \leq i \leq k_\mathbf{f}-1$, such that for $\mathbb{P}$-almost every $\omega$,
\[ \hspace{-1mm} \begin{array}{l l l}
\tilde{u}_n^i(\omega)\,:=\,\varphi_\mathbf{f}(n,\theta^{-n}\omega)(\tilde{u}^{i-nl_\mathbf{f}}(\theta^{-n}\omega)) \nenearrows a_i(\omega), & & \tilde{u}_{-n}^i(\omega)\,:=\,\varphi_\mathbf{f}(-n,\theta^n\omega)(\tilde{u}^{i+nl_\mathbf{f}}(\theta^n\omega)) \sesearrows r_{i-1}(\omega), \\
\hspace{0.3mm}\tilde{v}_n^i(\omega)\,:=\,\varphi_\mathbf{f}(n,\theta^{-n}\omega)(\tilde{v}^{i-nl_\mathbf{f}}(\theta^{-n}\omega)) \sesearrows a_i(\omega), & & \hspace{0.3mm}\tilde{v}_{-n}^i(\omega)\,:=\,\varphi_\mathbf{f}(-n,\theta^n\omega)(\tilde{v}^{i+nl_\mathbf{f}}(\theta^n\omega)) \nenearrows r_i(\omega).
\end{array}\]
\end{cor}

\begin{proof}
(A)~Taking $u$ as in Lemma~\ref{pull2}, we have that for $\mathbb{P}$-almost all $\omega$, $(u(\theta^n\omega))_{n \in \mathbb{Z}}$ is $(f_{\alpha_n})_{n \in \mathbb{Z}}$-generic; so let $\tilde{u}(\omega)$ be the $0$-coordinate of the anticlockwise perturbation of $(u(\theta^n\omega))$ by $(f_{\alpha_n})$. Note that for $\mathbb{P}$-almost all $\omega$, the anticlockwise perturbation of $(u(\theta^n\omega))$ by $(f_{\alpha_n})$ is $(\tilde{u}(\theta^n\omega))$. By Lemma~\ref{monotone pert} with $a_n=a_0(\theta^n\omega)$, we have that $\tilde{u}_n^0(\omega) \in \;]\tilde{u}_{n-1}^0(\omega),a_0(\omega)[$ for all $n \in \mathbb{Z}$, $\mathbb{P}$-almost surely; but also, for $\mathbb{P}$-almost all $\omega$, for infinitely many positive and negative $n$ we have that $\tilde{u}(\theta^n\omega)=u(\theta^n\omega)$ and so $\tilde{u}_n^0(\omega)=u_n(\omega)$ (where $u_n$ is as in Lemma~\ref{pull2}). Hence $\tilde{u}_n^0(\omega) \nenearrows a_0(\omega)$ and $\tilde{u}_{-n}^0(\omega) \sesearrows r_0(\omega)$. The construction and proof for $\tilde{v}$ is similar.
\\ \\
(B)~By Lemma~\ref{pull}, for $\mathbb{P}$-almost all $\omega$, for each $i$, $(\partial_-G_{i+nl_\mathbf{f}})_{n \in \mathbb{Z}}$ is $(f_{\alpha_n})_{n \in \mathbb{Z}}$-generic; so for each $i$, let $\tilde{u}^i(\omega)$ be the $0$-coordinate of the anticlockwise perturbation of $(\partial_-G_{i+nl_\mathbf{f}})$ by $(f_{\alpha_n})$. Note that for $\mathbb{P}$-almost all $\omega$, the anticlockwise perturbation of $(\partial_-G_{i+nl_\mathbf{f}})$ by $(f_{\alpha_n})$ is $(\tilde{u}^{i+nl_\mathbf{f}}(\theta^n\omega))$. By Lemma~\ref{pull} and Lemma~\ref{monotone pert} (replacing $a_n$ with $a_{i+nl_\mathbf{f}}(\theta^n\omega)$ and $x_n$ with $\partial_-G_{i+nl_\mathbf{f}}$), we have that $\tilde{u}_n^i(\omega) \in \;]\tilde{u}_{n-1}^i(\omega),a_i(\omega)[$ for all $n \in \mathbb{Z}$, $\mathbb{P}$-almost surely; but also, for $\mathbb{P}$-almost all $\omega$, for infinitely many positive and negative $n$ we have that $\tilde{u}^{i+nl_\mathbf{f}}(\theta^n\omega)=\partial_-G_{i+nl_\mathbf{f}}$ and so $\tilde{u}_n^i(\omega)=u_n^i(\omega)$ (where $u_n^i$ is as in Lemma~\ref{pull}). Hence $\tilde{u}_n^i(\omega) \nenearrows a_i(\omega)$ and $\tilde{u}_{-n}^i(\omega) \sesearrows r_{i-1}(\omega)$. The construction and proof for $\tilde{v}^i$ is similar.
\end{proof}

\section{Proof of main results}

To prove Theorem~\ref{MAIN}, it will be sufficient just to prove the characterisation of \emph{orientationally} conjugate dynamics, since the subsequent characterisation of \emph{topologically} conjugate dynamics then follows immediately by Remark~\ref{top op}.
\\ \\
We split the proof of Theorem~\ref{MAIN} into two cases: the ``generic'' case, where neither $\mathbf{f}$ nor $\mathbf{g}$ comes under category (c'); and the ``degenerate'' case, where at least one of $\mathbf{f}$ and $\mathbf{g}$ comes under category (c').

\subsection{Generic case}

An \emph{$\Omega$-based circle homeomorphism} is an $\Omega$-indexed family $(h_\omega)$ of orientation-preserving circle homeomorphisms $h_\omega \in \mathrm{Homeo}^+(\mathbb{S}^1)$ such that $\omega \mapsto h_\omega(x)$ is measurable for each $x \in \mathbb{S}^1$.
\\ \\
For each $k \in \mathbb{N}$ and $l \in \{0,\ldots,k-1\}$, define $g_{k,l} \in \mathrm{Homeo}^+(\mathbb{S}^1)$ by
\[ g_{k,l}([x]) \ = \ \left[ x \, + \; \tfrac{1}{2\pi k}\sin(2\pi kx)+\tfrac{l}{k} \right]. \]

\begin{thm} \label{genconj}
Let $\mathbf{f}$ be a random circle homeomorphism. In the case that $\mathbb{S}^1$ is $\mathbf{f}$-minimal, assume that $\mathbf{f}$ does not admit a symmetry. Let $k_\mathbf{f}$ and $l_\mathbf{f}$ be as in Proposition~\ref{permute}. Then there exists an $\Omega$-based circle homeomorphism $(h_\omega)$ such that for $\mathbb{P}$-almost all $\omega$,
\[ f_{\alpha_0} \ = \ h_{\theta\omega}^{-1} \ \circ \ g_{k_\mathbf{f},l_\mathbf{f}} \ \circ \ h_\omega \, . \]
\end{thm}

\begin{proof}
We will drop the subscript $_\mathbf{f}$ in $k_\mathbf{f}$ and $l_\mathbf{f}$. Let $\tilde{u}_n^i$, $a_i$ and $r_i$ be as in Corollary~\ref{strict pull}. For $\mathbb{P}$-almost all $\omega$, we will first define $h_\omega$ on $[\tilde{u}_n^i(\omega),\tilde{u}_{n+1}^i(\omega)]$ for $n \in \mathbb{Z}$ and $0 \leq i \leq k-1$. Note that
\[ \varphi_{\mathbf{f}}(-n,\omega)[\tilde{u}_n^i(\omega),\tilde{u}_{n+1}^i(\omega)] \ = \ [\tilde{u}_0^{i-nl}(\theta^{-n}\omega),\tilde{u}_1^{i-nl}(\theta^{-n}\omega)] \]
for all $n \in \mathbb{Z}$. For $\mathbb{P}$-almost all $\omega$, for each $i$, let $h_\omega$ map $[\tilde{u}_0^i(\omega),\tilde{u}_1^i(\omega)]$ linearly onto $[[\frac{4i+1}{4k}],g_{k,0}([\frac{4i+1}{4k}])]$, with the orientation preserved (i.e.\ $h_\omega(\tilde{u}_0^i(\omega))=[\frac{4i+1}{4k}]$). Then, for $\mathbb{P}$-almost all $\omega$, for all $n \in \mathbb{Z}$ and $0 \leq i \leq k-1$, define $h_\omega$ on $[\tilde{u}_n^i(\omega),\tilde{u}_{n+1}^i(\omega)]$ by
\[ h_\omega \big|_{[\tilde{u}_n^i(\omega),\tilde{u}_{n+1}^i(\omega)]} \ := \ g_{k,l}^n \;\circ\; h_{\theta^{-n}\omega} \big|_{[\tilde{u}_0^{i-nl}(\theta^{-n}\omega),\tilde{u}_1^{i-nl}(\theta^{-n}\omega)]} \;\circ\; \varphi_\mathbf{f}(-n,\omega). \]
This definition is indeed consistent at the endpoints of the intervals $[\tilde{u}_n^i(\omega),\tilde{u}_{n+1}^i(\omega)]$, with
\[ h_\omega(\tilde{u}_n^i(\omega)) \ = \ g_{k,0}^n([\tfrac{4i+1}{4k}]) \]
for all $n \in \mathbb{Z}$. Now define $h_\omega(r_i(\omega))=[\frac{i+1}{k}]$ and $h_\omega(a_i(\omega))=[\frac{2i+1}{2k}]$ for $\mathbb{P}$-almost all $\omega$. So $h_\omega$ is continuous on $[r_{i-1}(\omega),a_i(\omega)]$, mapping $[r_{i-1}(\omega),a_i(\omega)]$ bijectively onto $[[\frac{i}{k}],[\frac{2i+1}{2k}]]$.
\\ \\
Now, letting $\tilde{v}_n^i$ be as in Corollary~\ref{strict pull}, we can similarly construct $h_\omega$ on $[a_i(\omega),r_i(\omega)]$ such that
\[ h_\omega([\tilde{v}_{n+1}^i(\omega),\tilde{v}_n^i(\omega)]) \ = \ [ \, g_{k,0}^{n+1}([\tfrac{4i+3}{4k}]) \, , \, g_{k,0}^n([\tfrac{4i+3}{4k}]) \, ] \]
for all $n \in \mathbb{Z}$, $\mathbb{P}$-a.s.. Thus we have constructed $h_\omega$ on the whole of $\mathbb{S}^1$, and one can directly verify that $h_{\theta\omega} \circ \varphi_\mathbf{f}(1,\omega) = g_{k,l} \circ h_\omega$ for $\mathbb{P}$-almost all $\omega$.
\end{proof}

\begin{thm} \label{gendisc}
Suppose we have $k,l,k',l'$ and an $\Omega$-based circle homeomorphism $(h_\omega)$ such that
\[ g_{k,l} \ = \ h_{\theta\omega}^{-1} \ \circ \ g_{k',l'} \ \circ \ h_\omega \]
for $\mathbb{P}$-almost all $\omega$. Then $k=k'$ and $l=l'$.
\end{thm}

\begin{proof}
We first show that for $\mathbb{P}$-almost all $\omega$, $h_\omega$ maps $\mathcal{R}:=\{[\frac{i}{k}]\}_{i=0}^{k-1}$ bijectively onto $\mathcal{R}':=\{[\frac{i}{k'}]\}_{i=0}^{k'-1}$ (implying in particular that $k=k'$). For all $n$, we have that
\[ g_{k,l}^n \ = \ h_{\theta^n\omega}^{-1} \ \circ \ g_{k',l'}^n \ \circ \ h_\omega \]
for $\mathbb{P}$-almost all $\omega$. For any $\varepsilon \in (0,\frac{1}{k})$, letting $U_\varepsilon$ be the $\varepsilon$-neighbourhood of $\mathcal{R}$, we have that $\overline{g_{k,l}^n(U_\varepsilon)}$ converges to the whole circle as $n\to\infty$; so by the Poincar\'{e} recurrence theorem, for $\mathbb{P}$-almost all $\omega$ there is a sequence $m_n \to \infty$ such that $\overline{g_{k',l'}^{m_n}(h_\omega(U_\varepsilon))}$ converges to the whole circle as $n\to\infty$, implying in particular that $\mathcal{R}' \subset \overline{h_\omega(U_\varepsilon)}$. This holds for all $\varepsilon \in (0,\frac{1}{k})$, and so $\mathcal{R}' \subset h_\omega(\mathcal{R})$ for $\mathbb{P}$-almost all $\omega$. Applying the same argument to $(h_\omega^{-1})$ gives that $h_\omega(\mathcal{R}) \subset \mathcal{R}'$ for $\mathbb{P}$-almost all $\omega$.
\\ \\
So then, $k=k'$ and for $\mathbb{P}$-almost all $\omega$ there exists $m(\omega) \in \{0,\ldots,k-1\}$ such that $h_\omega([\frac{i}{k}])=[\frac{i+m(\omega)}{k}]$ for all $i$. Note that
\[ m(\theta\omega) \ = \ m(\omega)+(l'-l) \hspace{5mm} \mathrm{mod} \ k \]
and so
\[ m(\theta^k\omega) \ = \ m(\omega)  \hspace{5mm} \mathrm{mod} \ k. \]
Since $\theta^k$ is $\mathbb{P}$-ergodic, it follows that $m(\cdot)$ is constant almost everywhere ($\mathrm{mod} \ k$), and so $l=l'$.
\end{proof}

\noindent Combining Theorem~\ref{genconj} and Theorem~\ref{gendisc}, we obtain Theorem~\ref{MAIN} in the generic case.

\subsection{Degenerate case: preservation of the Antonov classification}

For any $k\in\mathbb{N}$, a \emph{random $k$-periodic point of $\varphi_\mathbf{f}$} is a measurable function $a \colon \Omega \to \mathbb{S}^1$ such that $a(\theta^k\omega)=\varphi_\mathbf{f}(k,\omega)a(\omega)$ for $\mathbb{P}$-almost all $\omega$.

\begin{lemma}
Let $\mathbf{f}$ and $\mathbf{g}$ be random circle homeomorphisms having orientationally conjugate dynamics, and suppose that $\mathbb{S}^1$ is $\mathbf{f}$-minimal. Then $\mathbb{S}^1$ is $\mathbf{g}$-minimal, and:
\begin{enumerate}[\indent (i)]
\item if $\mathbf{f}$ is deterministically orientationally conjugate to a random rotation, then so is $\mathbf{g}$;
\item given $m \geq 2$, if $\mathbf{f}$ admits an order-$m$ symmetry $\tau$ such that the $\tau$-lifts of $\mathbf{f}$ are contractive, then $\mathbf{g}$ also admits an order-$m$ symmetry $\tilde{\tau}$ such that the $\tilde{\tau}$-lifts of $\mathbf{g}$ are contractive.
\end{enumerate}
\end{lemma}

\begin{proof}
If $\mathbf{f}$ is deterministically orientationally conjugate to a random rotation, then by Lemma~\ref{rot inv}, there is a unique $\varphi_\mathbf{f}$-invariant measure; so since conjugacy preserves invariant measures, there is a unique $\varphi_\mathbf{g}$-invariant measure. Hence $\mathbf{g}$ cannot come under the categories addressed by Lemma~\ref{pull} or Lemma~\ref{sym contr}, so (using Proposition~\ref{antonov}) we must have that $\mathbb{S}^1$ is $\mathbf{g}$-minimal with $\mathbf{g}$ being deterministically orientationally conjugate to a random rotation.
\\ \\
Now suppose $\mathbf{f}$ admits an order-$m$ symmetry $\tau$ such that the $\tau$-lifts of $\mathbf{f}$ are contractive. First suppose for a contradiction that $\mathbb{S}^1$ is not $\mathbf{g}$-minimal. Letting $a_i,r_i$ ($0 \leq i \leq k_\mathbf{g}-1$) be as in Lemma~\ref{pull} with reference to $\mathbf{g}$, it is easy to see that $a_i,r_i$ are random $k_\mathbf{g}$-periodic points of $\mathbf{g}$; hence $\varphi_\mathbf{f}$ must also have random $k_\mathbf{g}$-periodic points. However, $\mathbb{S}^1$ is $\mathbf{f}^{k_\mathbf{g}}$-minimal by Proposition~\ref{n minimal}, and $\tau$ is obviously a symmetry of $\mathbf{f}^{k_\mathbf{g}}$; so by Lemma~\ref{sym contr}(B) applied to $\mathbf{f}^{k_\mathbf{g}}$, $\varphi_\mathbf{f}$ has no random $k_\mathbf{g}$-periodic point. Thus we have a contradiction, and so in fact $\mathbb{S}^1$ is $\mathbf{g}$-minimal. Now there are exactly two $\varphi_\mathbf{f}$-ergodic measures (by Lemma~\ref{sym contr}(B)), each consisting of $m$ points in the disintegration; so since conjugacy preserves invariant measures, the same must hold for $\mathbf{g}$. So $\mathbf{g}$ cannot be deterministically orientationally conjugate to a random rotation (Lemma~\ref{rot inv}), so by Proposition~\ref{antonov} and Lemma~\ref{sym contr}(B), $\mathbf{g}$ admits an order-$m$ symmetry $\tilde{\tau}$ such that the $\tilde{\tau}$-lifts of $\mathbf{g}$ are contractive.
\end{proof}

\noindent So to complete the proof of Theorem~\ref{MAIN}, it will be sufficient to prove the following results:

\begin{lemma} \label{hardconj}
Let $\mathbf{f}$ and $\mathbf{g}$ be random circle homeomorphisms such that $\mathbb{S}^1$ is $\mathbf{f}$-minimal and $\mathbf{g}$-minimal, and suppose that $\mathbf{f}$ and $\mathbf{g}$ have orientationally conjugate dynamics.
\begin{enumerate}[\indent (A)]
\item If $\mathbf{f}$ and $\mathbf{g}$ are random rotations, then $\mathbf{f}=\mathbf{g}$.
\item Given $m \geq 2$, if $\tau_m$ is a symmetry of both $\mathbf{f}$ and $\mathbf{g}$, with $\mathbf{F}=(z_m(f_\alpha))$ and $\mathbf{G}=(z_m(f_\alpha))$ being contractive, then $\mathbf{f}$ and $\mathbf{g}$ are deterministically topologically conjugate; moreover, in the case that $m \geq 3$, $\mathbf{f}$ and $\mathbf{g}$ are deterministically orientationally conjugate.
\end{enumerate}
\end{lemma}

\begin{lemma} \label{revconj}
Let $\mathbf{f}$ be a random circle homeomorphism such that $\mathbb{S}^1$ is $\mathbf{f}$-minimal, and suppose that $\tau_2$ is a symmetry of $\mathbf{f}$ with $\mathbf{F}=(z_2(f_\alpha))$ being contractive. Let $j \in \mathrm{Homeo}^-(\mathbb{S}^1)$ be an orientation-reversing homeomorphism such that $\tau_2$ is also a symmetry of the random circle homeomorphism $\mathbf{g}:=(j \circ f_\alpha \circ j^{-1})$. Then $\mathbf{f}$ and $\mathbf{g}$ have orientationally conjugate dynamics.
\end{lemma}

\subsection{Proof of Lemma~\ref{hardconj}(A)}

For any $c \in \mathbb{S}^1$, define $R_c \colon \mathbb{S}^1 \to \mathbb{S}^1$ to be the rotation $R_c(x)=x+c$. Define the equivalence relation $\sim$ on $\mathrm{Homeo}^+(\mathbb{S}^1)$ by
\[ f \sim g \hspace{3mm} \Longleftrightarrow \hspace{3mm} \exists \, a,b \in \mathbb{S}^1 \textrm{ s.t.~} g \, = \, R_a \circ f \circ R_{b\,}. \]
\noindent Let $\mathcal{I}$ be the set of all equivalence classes of $\sim\,$. For any $f \in \mathrm{Homeo}^+(\mathbb{S}^1)$, let $\hat{f} \in \mathcal{I}$ be the equivalence class represented by $f$. We will say that a set $A \subset \mathrm{Homeo}^+(\mathbb{S}^1)$ is \emph{$\sim$-invariant} if for every $f \in A$, we have $\hat{f} \subset A$.

\begin{lemma} \label{quotient}
Let $\mathcal{P}$ be a Borel probability measure on $\mathrm{Homeo}^+(\mathbb{S}^1)$. If $\mathcal{P}(A) \in \{0,1\}$ for every $\sim$-invariant open set $A \subset \mathrm{Homeo}^+(\mathbb{S}^1)$, then there exists $\xi \in \mathcal{I}$ such that $\mathcal{P}(\xi)=1$.
\end{lemma}

\begin{proof}
Since $\mathrm{Homeo}^+(\mathbb{S}^1)$ is a topological group, the map $(a,b,f) \mapsto R_a \circ f \circ R_b$ is continuous. Hence in particular, each element of $\mathcal{I}$ is a compact subset of $\mathrm{Homeo}^+(\mathbb{S}^1)$. Now suppose there does not exist $\xi \in \mathcal{I}$ such that $\mathcal{P}(\xi)=1$. Fix any $f_0 \in \mathrm{supp}\,\mathcal{P}$. Let $h$ be an element of $(\mathrm{supp}\,\mathcal{P}) \setminus \hat{f}_0$. Since $\mathbb{S}^1$ is compact and $R_a \circ f_0 \circ R_b \neq h$ for all $a,b \in \mathbb{S}^1$, there are neighbourhoods $U$ of $f_0$ and $V$ of $h$ such that $h' \nin \hat{f}$ for all $h' \in V$ and $f \in U$. Note that the set $A:=\bigcup_{f \in U} \hat{f}$ is open, since the map $f \mapsto R_a \circ f \circ R_b$ is a homeomorphism of $\mathrm{Homeo}^+(\mathbb{S}^1)$ for each $a,b \in \mathbb{S}^1$. Obviously $A$ is $\sim$-invariant; but since both $f_0$ and $h$ are in $\mathrm{supp}\,\mathcal{P}$, we have that $\mathcal{P}(A) \nin \{0,1\}$.
\end{proof}

\noindent Now write $\mathbf{f}=(R_{s(\alpha)})$ and $\mathbf{g}=(R_{s'(\alpha)})$, and let $(h_\omega)$ be an orientation-preserving conjugacy from $\varphi_\mathbf{f}$ to $\varphi_\mathbf{g}$. So
\begin{equation} \label{rot conj}
R_{s(\alpha_0)} \ = \ h_{\theta\omega}^{-1} \,\circ\, R_{s'(\alpha_0)} \,\circ\, h_\omega
\end{equation}
\noindent for $\mathbb{P}$-almost all $\omega \in \Omega$. Hence $\hat{h}_{\theta\omega}=\hat{h}_\omega$ for $\mathbb{P}$-almost all $\omega \in \Omega$. So since $\mathbb{P}$ is $\theta$-ergodic, by Lemma~\ref{quotient} there exists $\xi \in \mathcal{I}$ such that for $\mathbb{P}$-almost all $\omega \in \Omega$, $\hat{h}_\omega=\xi$. Let $\tilde{\Omega}:=\{\omega:\hat{h}_\omega=\xi\}$. Fix an arbitrary $f \in \xi$. By a suitable ``measurable selection'' theorem (e.g.~\cite[Corollary~A.6]{MR776417}), there exist measurable functions $a,b \colon \tilde{\Omega} \to \mathbb{S}^1$ such that for all $\omega \in \tilde{\Omega}$, $h_\omega=R_{a(\omega)} \circ f \circ R_{b(\omega)}$. With this, (\ref{rot conj}) can be rearranged as
\[ R_{s(\alpha_0) \,+\, b(\theta\omega) \,-\, b(\omega)} \ = \ f^{-1} \,\circ\, R_{- a(\theta\omega) \,+\, s'(\alpha_0) \,+\, a(\omega)} \,\circ\, f. \]
\noindent Hence, for $\mathbb{P}$-almost all $\omega \in \tilde{\Omega}$,
\[ s(\alpha_0) + b(\theta\omega) - b(\omega) \ = \ - a(\theta\omega) + s'(\alpha_0) + a(\omega) \]
\noindent and so
\[ a(\theta\omega)+b(\theta\omega) \ = \  s'(\alpha_0) - s(\alpha_0) \ + \ a(\omega)+b(\omega). \]
\noindent Hence, by Corollary~\ref{rot rfp} with $\kappa=a+b$, we have that $s'-s=[0]$ everywhere, i.e.\ $s=s'$.

\subsection{Symmetrisation of conjugacy and the proof of Lemma~\ref{revconj}}

The following lemma immediately gives the proof of Lemma~\ref{revconj}, and will also be used in the proof of Lemma~\ref{hardconj}(B).

\begin{lemma} \label{sym of conj}
Fix $m \in \mathbb{N}$. Let $\mathbf{f}$ and $\mathbf{g}$ be random circle homeomorphisms such that $\mathbb{S}^1$ is $\mathbf{f}$-minimal and $\mathbf{g}$-minimal. Suppose that $\tau_m$ is a symmetry of both $\mathbf{f}$ and $\mathbf{g}$, and $\mathbf{F}=(z_m(f_\alpha))$ and $\mathbf{G}=(z_m(g_\alpha))$ are contractive, with attractors $A_\mathbf{f}$ and $A_\mathbf{g}$ respectively. For each $\omega$, let $a_\mathbf{f}(\omega)$ (resp.\ $a_\mathbf{g}(\omega)$) be the unique value in $[[0],[\frac{1}{m}][$ such that $ma_\mathbf{f}(\omega)=A_\mathbf{f}(\omega)$ (resp.\ $ma_\mathbf{g}(\omega)=A_\mathbf{g}(\omega)$). Suppose there is an $\Omega$-indexed family $(\tilde{h}_\omega)$ of functions $\tilde{h}_\omega \colon \mathbb{S}^1 \to \mathbb{S}^1$ (\emph{not} necessarily in $\mathrm{Homeo}^+(\mathbb{S}^1)$) such that $(\omega,x) \mapsto \tilde{h}_\omega(x)$ is measurable and for $\mathbb{P}$-almost every $\omega$:
\begin{itemize}
\item there is an integer $n_\omega$ such that for each $0 \leq i \leq m-1$,
\[ \tilde{h}_\omega(a_\mathbf{f}(\omega) + [\tfrac{i}{m}]) \ = \ a_\mathbf{g}(\omega) + [\tfrac{n_\omega+i}{m}] \, ; \]
\item $\tilde{h}_{\theta\omega}(f_{\alpha_0}(a_\mathbf{f}(\omega))) \ = \ g_{\alpha_0}(\tilde{h}_\omega(a_\mathbf{f}(\omega)))$.
\end{itemize}
Then there exists an orientation-preserving conjugacy $(h_\omega)$ from $\varphi_\mathbf{f}$ to $\varphi_\mathbf{g}$ such that $h_\omega$ commutes with $\tau_m$ for all $\omega$.
\end{lemma}

\begin{proof}
By Theorem~\ref{genconj}, there is an orientation-preserving conjugacy $(H_\omega)$ from $\varphi_\mathbf{F}$ to $\varphi_\mathbf{G}$. For $\mathbb{P}$-almost all $\omega$, we have (by Corollary~\ref{pres attr}) that
\[ H_\omega(A_\mathbf{f}(\omega)) \ = \ A_\mathbf{g}(\omega) \ = \ m\tilde{h}_\omega(a_\mathbf{f}(\omega)); \]
so let $h_\omega$ be the unique orientation-preserving homeomorphism commuting with $\tau_m$ such that $z_m(h_\omega)=H_\omega$ and $h_\omega(a_\mathbf{f}(\omega))=\tilde{h}_\omega(a_\mathbf{f}(\omega))$. We need to show that $(h_\omega)$ is a conjugacy from $\varphi_\mathbf{f}$ to $\varphi_\mathbf{g}$. Note that
\begin{equation} \label{homega}
h_\omega(a_\mathbf{f}(\omega)+[\tfrac{i}{m}]) \ = \ \tilde{h}_\omega(a_\mathbf{f}(\omega)+[\tfrac{i}{m}]).
\end{equation}
for each $i$. We have
\[ z_m(h_{\theta\omega} \circ f_{\alpha_0}) \ = \ H_{\theta\omega} \circ z_m(f_{\alpha_0}) \ = \ z_m(g_{\alpha_0}) \circ H_\omega \ = \ z_m(g_{\alpha_0} \circ h_\omega), \]
\noindent and so there exists $i_\omega \in \{0,\ldots,m-1\}$ such that
\[ h_{\theta\omega} \circ f_{\alpha_0} \ = \ \tau_m^{i_\omega} \circ g_{\alpha_0} \circ h_\omega. \]
\noindent Now $f_{\alpha_0}(a_\mathbf{f}(\omega)) - a_\mathbf{f}(\theta\omega)$ is a multiple of $[\frac{1}{m}]$; so since $h_{\theta\omega}(a_\mathbf{f}(\theta\omega))=\tilde{h}_{\theta\omega}(a_\mathbf{f}(\theta\omega))$, we have that
\begin{align*}
h_{\theta\omega} \circ f_{\alpha_0}(a_\mathbf{f}(\omega)) \ &= \ \tilde{h}_{\theta\omega} \circ f_{\alpha_0}(a_\mathbf{f}(\omega))  \hspace{4mm} \textrm{(by (\ref{homega}) applied to $\theta\omega$)} \\
&= \ g_{\alpha_0} \circ \tilde{h}_\omega(a_\mathbf{f}(\omega)) \\
&= \ g_{\alpha_0} \circ h_\omega(a_\mathbf{f}(\omega)).
\end{align*}
\noindent So $i_\omega=0$. So we are done.
\end{proof}

\noindent Now given a random circle homeomorphism $\mathbf{f}$ as in Lemma~\ref{sym of conj} and an orientation-reversing homeomorphism $j$, if $\mathbf{g}:=(j \circ f_\alpha \circ j^{-1})$ also has $\tau_m$ as a symmetry, then for $\mathbb{P}$-almost all $\omega$ there exists $n_\omega$ such that
\[ j(a_\mathbf{f}(\omega) + [\tfrac{i}{m}]) \ = \ a_\mathbf{g}(\omega) + [\tfrac{n_\omega-i}{m}] \]
\noindent for all $i$. But if $m=2$ then $[\frac{-i}{m}]=[\frac{i}{m}]$, and so Lemma~\ref{revconj} is obtained simply by applying Lemma~\ref{sym of conj} with $\tilde{h}_\omega=j$ for all $\omega$.

\subsection{Further preparation for the proof of Lemma~\ref{hardconj}(B)} \label{Further}

Fix any $m \in \mathbb{N}$. Let $\mathcal{J} = \{h \in \mathrm{Homeo}^+(\mathbb{S}^1) : \textrm{$h$ commutes with $\tau_m$} \}$. For any $x \in \mathbb{S}^1$, let $\breve{x}$ be the unique lift of $x$ in $[0,1)$, let $\zeta(x) \in \{0,\ldots,m-1\}$ be such that $\breve{x} \in [\frac{\zeta(x)}{m},\frac{\zeta(x)+1}{m})$, and let $d(x)$ be the unique value in $[[0],[\frac{1}{m}][$ such that $md(x)=x$. For any $h \in \mathcal{J}$ and $x \in \mathbb{S}^1$, let $L(h,x):=\zeta(h(d(x)))$.
\\ \\
For any $\mathbf{x}_1=(x_1,y_1)$ and $\mathbf{x}_2=(x_2,y_2)$ in $(\mathbb{S}^1)^2$ with $x_1 \neq x_2$ and $y_1 \neq y_2$, define
\[ \mathcal{H}_{\{\mathbf{x}_1,\mathbf{x}_2\}} \ := \ \{ H \in \mathrm{Homeo}^+(\mathbb{S}^1) \, : \, H(x_1) = y_1, \, H(x_2) = y_2 \} \]
\noindent and let
\[ v(\mathbf{x}_1,\mathbf{x}_2) \ = \ \left\{ \begin{array}{c l}
0 & \breve{x}_1 < \breve{x}_2 \Leftrightarrow \breve{y}_1 < \breve{y}_2 \\
1 & \breve{x}_1 < \breve{x}_2 \textrm{ and } \breve{y}_1 > \breve{y}_2 \\
-1 & \breve{x}_1 > \breve{x}_2 \textrm{ and } \breve{y}_1 < \breve{y}_2.
\end{array} \right. \]

\noindent Note that $\mathcal{H}_{\{\mathbf{x}_1,\mathbf{x}_2\}}$ is non-empty. Also note that $v(\mathbf{x}_1,\mathbf{x}_2)=-v(\mathbf{x}_2,\mathbf{x}_1)$.

\begin{lemma} \label{vect}
For any $h \in \mathcal{J}$ with $z_m(h) \in \mathcal{H}_{\{\mathbf{x}_1,\mathbf{x}_2\}}$,
\[ L(h,x_2) \ = \ L(h,x_1) + v(\mathbf{x}_1,\mathbf{x}_2) \hspace{5mm} \mathrm{mod} \ m. \]
\end{lemma}

\begin{proof}
Since $v$ is antisymmetric, we can assume without loss of generality that $\breve{x}_1<\breve{x}_2$. Fix any $h \in \mathcal{J}$ with $z_m(h) \in \mathcal{H}_{\{\mathbf{x}_1,\mathbf{x}_2\}}$, and let $\mathfrak{h} \colon \mathbb{R} \to \mathbb{R}$ be a lift of $h$. Note that
\begin{align*}
\mathfrak{h}\left(\frac{\breve{x}_1}{m}\right) \ &= \ \frac{\breve{y}_1 + L(h,x_1)}{m} + c_1 \\
\mathfrak{h}\left(\frac{\breve{x}_2}{m}\right) \ &= \ \frac{\breve{y}_2 + L(h,x_2)}{m} + c_2
\end{align*}
\noindent for some integers $c_1,c_2 \in \mathbb{Z}$. So, since $h \in \mathcal{J}$ and $\frac{\breve{x}_2}{m} \in (\frac{\breve{x}_1}{m},\frac{\breve{x}_1+1}{m})$, it follows that
\[ \frac{\breve{y}_1 + L(h,x_1)}{m} + c_1 \ < \ \frac{\breve{y}_2 + L(h,x_2)}{m} + c_2 \ < \ \frac{\breve{y}_1 + L(h,x_1)}{m} + c_1 + \frac{1}{m} \]
\noindent and therefore
\[ \breve{y}_1 \ < \ \breve{y}_2 + \underbrace{L(h,x_2) - L(h,x_1) + m(c_2-c_1)}_{r} \ < \ \breve{y}_1 + 1. \]
\noindent If $\breve{y}_1<\breve{y}_2$ (i.e.~if $v(\mathbf{x}_1,\mathbf{x}_2)=0$) then we must have that $\breve{y}_2 \in (\breve{y}_1,\breve{y}_1+1)$, and so $r=0$. Likewise, if $\breve{y}_1>\breve{y}_2$ (i.e.~if $v(\mathbf{x}_1,\mathbf{x}_2)=1$) then we must have that $\breve{y}_2 \in (\breve{y}_1-1,\breve{y}_1)$, and so $r=1$.
\end{proof}

\noindent Let $\mathbf{S}_3$ be the set of 3-element subsets $\{(x_1,y_1),(x_2,y_2),(x_3,y_3)\}$ of $(\mathbb{S}^1)^2$ such that $x_1 \neq x_2 \neq x_3 \neq x_1$ and $y_1 \neq y_2 \neq y_3 \neq y_1$. Note that for any $\mathbf{x}_1=(x_1,y_1)$, $\mathbf{x}_2=(x_2,y_2)$ and $\mathbf{x}_3=(x_3,y_3)$ with $\{\mathbf{x}_1,\mathbf{x}_2,\mathbf{x}_3\} \in \mathbf{S}_3$, the following two statements are equivalent:
\begin{enumerate}[\indent (i)]
\item there exists $H \in \mathrm{Homeo}^+(\mathbb{S}^1)$ such that $H(x_1)=y_1$, $H(x_2)=y_2$, and $H(x_3)=y_3$;
\item $x_2 \in [x_1,x_3] \Leftrightarrow y_2 \in [y_1,y_3]$.
\end{enumerate}
Let
\[ \mathbf{V} \ := \ \{ \, \{(x_1,y_1),(x_2,y_2),(x_3,y_3)\} \in \mathbf{S}_3 \, : \, x_2 \in [x_1,x_3] \Leftrightarrow y_2 \in [y_1,y_3] \, \}. \]
For any $\mathbf{x}_1,\mathbf{x}_2,\mathbf{x}_3 \in (\mathbb{S}^1)^2$ with $\{\mathbf{x}_1,\mathbf{x}_2,\mathbf{x}_3\} \in \mathbf{S}_3$, define
\[ \mathbf{v}(\mathbf{x}_1,\mathbf{x}_2,\mathbf{x}_3) \ := \ v(\mathbf{x}_1,\mathbf{x}_2) + v(\mathbf{x}_2,\mathbf{x}_3) - v(\mathbf{x}_1,\mathbf{x}_3). \]

\begin{lemma} \label{triangle}
$\mathbf{v}(\mathbf{x}_1,\mathbf{x}_2,\mathbf{x}_3)$ is equal to either $-1$, $0$ or $1$. Moreover, $\mathbf{v}(\mathbf{x}_1,\mathbf{x}_2,\mathbf{x}_3)=0$ if and only if $\{\mathbf{x}_1,\mathbf{x}_2,\mathbf{x}_3\} \in \mathbf{V}$.
\end{lemma}

\begin{proof}
If we swap $\mathbf{x}_1$ and $\mathbf{x}_2$, we negate the value of $\mathbf{v}$:
\[ v(\mathbf{x}_2,\mathbf{x}_1) + v(\mathbf{x}_1,\mathbf{x}_3) - v(\mathbf{x}_2,\mathbf{x}_3) \ = \ -v(\mathbf{x}_1,\mathbf{x}_2) - v(\mathbf{x}_2,\mathbf{x}_3) + v(\mathbf{x}_1,\mathbf{x}_3). \]
\noindent Likewise for if we swap $\mathbf{x}_3$ and $\mathbf{x}_2$. Since the permutations $(1 \ 2)$ and $(2 \ 3)$ generate the permutation group on $\{1,2,3\}$, it follows that every permutation of $\mathbf{x}_1$, $\mathbf{x}_2$ and $\mathbf{x}_3$ either preserves or negates the value of $\mathbf{v}$. Hence, we may assume without loss of generality that $\breve{x}_1<\breve{x}_2<\breve{x}_3$.
\\ \\
If $\breve{y}_1<\breve{y}_3$, then $v(\mathbf{x}_1,\mathbf{x}_3)=0$ and so it is easy to check that
\[ \mathbf{v}(\mathbf{x}_1,\mathbf{x}_2,\mathbf{x}_3) \ = \ \left\{ \begin{array}{c l} 0 & \breve{y}_2 \in (\breve{y}_1,\breve{y}_3) \\ 1 & \textrm{otherwise.} \end{array} \right. \]
\noindent If $\breve{y}_1>\breve{y}_3$, then $v(\mathbf{x}_1,\mathbf{x}_3)=1$ and so it is easy to check that
\[ \mathbf{v}(\mathbf{x}_1,\mathbf{x}_2,\mathbf{x}_3) \ = \ \left\{ \begin{array}{c l} 1 & \breve{y}_2 \in (\breve{y}_3,\breve{y}_1) \\ 0 & \textrm{otherwise.} \end{array} \right. \]
\noindent It is easy to check from this that $\mathbf{v}(\mathbf{x}_1,\mathbf{x}_2,\mathbf{x}_3)=0$ if and only if $(\mathbf{x}_1,\mathbf{x}_2,\mathbf{x}_3) \in \mathbf{V}$.
\end{proof}

\subsection{Proof of Lemma~\ref{hardconj}(B)}

\begin{rmk} \label{assumption}
Observe that for any rotation map $R_c \colon x \mapsto x+c$, $\,\tau_m$ is a symmetry of the random circle homeomorphisms $(R_c \circ f_\alpha \circ R_{-c})$ and $(R_c \circ g_\alpha \circ R_{-c})$, and $z_m(R_c)=R_{mc}$; also note that any orientation-reversing homeomorphism has a fixed point. Hence, without loss of generality, it is possible to make the following assumption that we will use later:
\begin{enumerate}[\indent (a)]
\item if $m=3$ and there is an orientation-reversing deterministic conjugacy from $\mathbf{F}$ to $\mathbf{G}$, then there is an orientation-reversing deterministic conjugacy $K$ from $\mathbf{F}$ to $\mathbf{G}$ such that $K([0])=[0]$;
\item if $m=4$ and there is an orientation-reversing deterministic conjugacy from $\mathbf{F}$ to $\mathbf{G}$, then there is an orientation-reversing deterministic conjugacy $K$ from $\mathbf{F}$ to $\mathbf{G}$ such that $K([0]) \neq [0]$.
\end{enumerate}
\end{rmk}

\noindent Now let $\Omega^+:=\Delta^{\mathbb{N}_0}$, with $\mathbb{P}^+:=\nu^{\otimes \mathbb{N}_0}$; and let $\Omega^-:=\Delta^{-\mathbb{N}}$, with $\mathbb{P}^-:=\nu^{\otimes (-\mathbb{N})}$. An element of $\Omega^+$ will typically be denoted $\omega^+$, and an element of $\Omega^-$ will typically be denoted $\omega^-$. If we are given an element $\omega=(\alpha_i)_{i \in \mathbb{Z}}$ of $\Omega$, we will write $\omega^+:=(\alpha_i)_{i \geq 0}$ and $\omega^-:=(\alpha_i)_{i < 0}$. By convention, we will identify $\Omega$ with $\Omega^+ \times \Omega^-$ (\emph{not} with $\Omega^- \times \Omega^+$). For any function $F$ with domain $\Omega^-$, we define the function $\mathring{F}$ with domain $\Omega$ by $\mathring{F}(\omega)=F(\omega^-)$.
\\ \\
For any random circle homeomorphism $\mathbf{k}$ such that $\mathbb{S}^1$ is $\mathbf{k}$-minimal, we will write $\rho_\mathbf{k}$ for the unique $\mathbf{k}$-stationary measure.
\\ \\
Now let $A_\mathbf{f} \colon \Omega^- \to \mathbb{S}^1$ and $R_\mathbf{f} \colon \Omega^+ \to \mathbb{S}^1$ be such that the attractor of $\mathbf{F}$ is the map $\mathring{A}_\mathbf{f} \colon \omega \mapsto A_\mathbf{f}(\omega^-)$, and the repeller of $\mathbf{F}$ is the map $\omega \mapsto R_\mathbf{f}(\omega^+)$. Define $A_\mathbf{g}$ and $R_\mathbf{g}$ likewise for $\mathbf{G}$. Let $a_\mathbf{f}(\omega^-)$ (resp.\ $a_\mathbf{g}(\omega^-)$, $r_\mathbf{f}(\omega^+)$, $r_\mathbf{g}(\omega^+)$) be the unique point in $[[0],[\frac{1}{m}][$ whose $m$-th multiple is $A_\mathbf{f}(\omega^-)$ (resp.\ $A_\mathbf{g}(\omega^-)$, $R_\mathbf{f}(\omega^+)$, $R_\mathbf{g}(\omega^+)$).
\\ \\
By Corollary~\ref{pres attr} and Lemma~\ref{sym of conj}, there exists a conjugacy $(h_\omega)$ from $\varphi_\mathbf{f}$ to $\varphi_\mathbf{g}$ such that $h_\omega$ is an orientation-preserving homeomorphism commuting with $\tau_m$ for all $\omega$. Let $H_\omega:=z_m(h_\omega)$.
\\ \\
Let $\zeta(\cdot)$, $L(\cdot,\cdot)$, $v(\cdot,\cdot)$, $\mathbf{S}_3$, $\mathbf{V}$ and $\mathbf{v}(\cdot,\cdot,\cdot)$ be as in Section~\ref{Further}. Define $\tilde{T}^-,\tilde{T}^+\colon \Omega \to \{0,\ldots,m-1\}$ by
\begin{align*}
\tilde{T}^-(\omega) \ &= \ \zeta(h_\omega(a_\mathbf{f}(\omega^-))) \ = \ L(h_\omega,A_\mathbf{f}(\omega^-)) \\
\tilde{T}^+(\omega) \ &= \ \zeta(h_\omega(r_\mathbf{f}(\omega^+))) \ = \ L(h_\omega,R_\mathbf{f}(\omega^+)).
\end{align*}
\noindent So for $\mathbb{P}$-almost all $\omega$,
\begin{align*}
h_\omega(a_\mathbf{f}(\omega^-)) - a_\mathbf{g}(\omega^-) \ &= \ [\tfrac{\tilde{T}^-(\omega)}{m}] \\
h_\omega(r_\mathbf{f}(\omega^+)) - r_\mathbf{g}(\omega^+) \ &= \ [\tfrac{\tilde{T}^+(\omega)}{m}].
\end{align*}

\begin{lemma}
$\tilde{T}^-$ is measurable with respect to the $\mathbb{P}$-completion of $\mathcal{F}_-$, and $\tilde{T}^+$ is measurable with respect to the $\mathbb{P}$-completion of $\mathcal{F}_+$.
\end{lemma}

\begin{proof}
We will just show the first statement; the second statement is proved similarly. Define $\beta \colon \Delta \times \Omega^- \to \{0,\ldots,m-1\}$ by
\[ \beta(\alpha,\omega^-) \ = \ \zeta(g_\alpha(a_\mathbf{g}(\omega^-))) - \zeta(f_\alpha(a_\mathbf{f}(\omega^-))) \hspace{5mm} \mathrm{mod} \ m \]
and define $\Theta_\beta \colon \Omega \times \{0,\ldots,m-1\} \to \Omega \times \{0,\ldots,m-1\}$ by
\[ \Theta_\beta(\omega,x) \ = \ ( \, \theta\omega \, , \, x + \beta(\alpha_0,\omega^-) \ \mathrm{mod} \ m \, ). \]
A straightforward computation shows that for $\mathbb{P}$-almost all $\omega$,
\[ \tilde{T}^-(\omega) + \beta(\alpha_0,\omega^-) \ = \ \tilde{T}^-(\theta\omega) \hspace{5mm} \mathrm{mod} \ m. \]
\noindent So letting $T_i:=\tilde{T}^{-\!}+i \ \mathrm{mod}\ m\,$ for each $0 \leq i \leq m-1$, the $\Theta_\beta$-ergodic measures with $\Omega$-marginal $\mathbb{P}$ are precisely the measures $(\delta_{T_i(\omega)})_{\omega \in \Omega}$. The method that we now employ is based on ideas in \cite{Cra07} and \cite{HO07}: Define the random map $\mathbf{B}$ on $\Omega^- \times \{0,\ldots,m-1\}$ by
\[ B_{\alpha_0}( \omega^- , \, x \, ) \ = \ ( \, (\alpha_{i+1})_{i<0} \, , \, x + \beta(\alpha_0,\omega^-) \ \mathrm{mod} \ m \, ) \]
\noindent for all $\alpha_0 \in \Delta$, $\,\omega^-\!=\!(\alpha_i)_{i<0} \in \Omega^-$, and $x \in \{0,\ldots,m-1\}$. Note that for all $\omega \in \Omega$ and $x \in \{0,\ldots,m-1\}$,
\begin{equation} \label{betaskew}
\Theta_\beta(\omega,x) \ = \ ( \, \theta^+\omega^+ \, , \, B_{\alpha_0}(\omega^-,x) \, )
\end{equation}
\noindent where $\theta^+$ is the shift map on $\Omega^+$. The measure $\mathbb{P}^- \otimes \frac{1}{m}(\delta_0 + \ldots + \delta_{m-1})$ is $\mathbf{B}$-stationary, and so there is at least one $\mathbf{B}$-ergodic measure $q$. Due to (\ref{betaskew}), \cite[Theorem~I.2.1]{Kif86} gives that $\mathbb{P}^+ \otimes q$ is an ergodic probability measure of $\Theta_\beta$. But note that any $\mathbf{B}$-stationary measure must have $\Omega^-$-marginal $\mathbb{P}^-$; hence $\mathbb{P}^+ \otimes q$ has $\Omega$-marginal $\mathbb{P}$, and so $\mathbb{P}^+ \otimes q$ has the $\Omega$-disintegration $(\delta_{T_i(\omega)})$ for some $i$. So letting $(q_u)_{u \in \Omega^-}$ be the $\Omega^-$-disintegration of $q$, we have that for $\mathbb{P}$-almost all $\omega \in \Omega$, $\delta_{T_i(\omega)}=q_{\omega^-}$.
\end{proof}

\noindent So let $\tilde{\Omega}$ be the set of sample points $\omega$ for which the following statements all hold:
\begin{itemize}
\item $H_\omega(A_\mathbf{f}(\omega^-))=A_\mathbf{g}(\omega^-)$
\item $H_\omega(R_\mathbf{f}(\omega^+))=R_\mathbf{g}(\omega^+)$
\item $\tilde{T}^-(\omega)=T^-(\omega^-)$
\item $\tilde{T}^+(\omega)=T^+(\omega^+)$.
\end{itemize}
where $T^- \colon \Omega^- \to \{0,\ldots,m-1\}$ and $T^+ \colon \Omega^+ \to \{0,\ldots,m-1\}$ are measurable functions. For each $\omega^- \in \Omega^-$, let
\[ \tilde{\Omega}_{\omega^-} \ := \ \{ \omega^+ \in \Omega^+ : (\omega^+,\omega^-) \in \tilde{\Omega} \}. \]
Let
\[ \tilde{\Omega}^- \ := \ \{\omega^- \in \Omega^- : \tilde{\Omega}_{\omega^-} \textrm{ is a $\mathbb{P}^+$-full set} \}. \]

\noindent Define the measure $\mathfrak{r}$ on $(\mathbb{S}^1)^2$ to be the pushforward of $\mathbb{P}^+$ under the map $\omega^+ \mapsto (R_\mathbf{f}(\omega^+),R_\mathbf{g}(\omega^+))$.

\begin{lemma} \label{redund}
Fix any $\omega_1^-,\omega_2^- \in \tilde{\Omega}^-$ and let
\begin{align*}
\mathbf{x}_1 \ &= \ (A_\mathbf{f}(\omega_1^-),A_\mathbf{g}(\omega_1^-)) \\
\mathbf{x}_2 \ &= \ (A_\mathbf{f}(\omega_2^-),A_\mathbf{g}(\omega_2^-)).
\end{align*}
For any $\mathbf{y},\mathbf{y}' \in \mathrm{supp}\,\mathfrak{r}$, if $\{\mathbf{x}_1,\mathbf{x}_2,\mathbf{y}\},\{\mathbf{x}_1,\mathbf{x}_2,\mathbf{y}'\} \in \mathbf{S}_3$, then
\[ \{\mathbf{x}_1,\mathbf{x}_2,\mathbf{y}\} \in \mathbf{V} \hspace{3mm} \Longleftrightarrow \hspace{3mm} \{\mathbf{x}_1,\mathbf{x}_2,\mathbf{y}'\} \in \mathbf{V}. \]
\end{lemma}

\begin{proof}
Suppose we have $\mathbf{y} \in \mathrm{supp}\,\mathfrak{r}$ such that $\{\mathbf{x}_1,\mathbf{x}_2,\mathbf{y}\} \in \mathbf{S}_3$. Let $U \subset (\mathbb{S}^1)^2$ be a neighbourhood of $\mathbf{y}$ sufficiently small that for all $\tilde{\mathbf{y}} \in U$, we have $\{\mathbf{x}_1,\mathbf{x}_2,\tilde{\mathbf{y}}\} \in \mathbf{S}_3$ and
\[ \{\mathbf{x}_1,\mathbf{x}_2,\mathbf{y}\} \in \mathbf{V} \hspace{3mm} \Longleftrightarrow \hspace{3mm} \{\mathbf{x}_1,\mathbf{x}_2,\tilde{\mathbf{y}}\} \in \mathbf{V}. \]
\noindent Since $\mathfrak{r}(U)>0$, there exists $\omega^+ \in \tilde{\Omega}_{\omega_1^-} \cap \tilde{\Omega}_{\omega_2^-}$ such that $(R_\mathbf{f}(\omega^+),R_\mathbf{g}(\omega^+)) \in U$. Using Lemma~\ref{vect}, a straightforward computation gives that
\[ \mathbf{v}(\mathbf{x}_1,\mathbf{x}_2,R_\mathbf{f}(\omega^+),R_\mathbf{g}(\omega^+)) \ = \ v(\mathbf{x}_1,\mathbf{x}_2) \, + \, T^-(\omega_1^-) \, - \, T^-(\omega_2^-) \hspace{5mm} \mathrm{mod} \ m. \]
\noindent So by Lemma~\ref{triangle}, since $m \geq 2$, we have that $\{\mathbf{x}_1,\mathbf{x}_2,\mathbf{y}\} \in \mathbf{V}$ if and only if
\begin{equation} \label{redeq}
v(\mathbf{x}_1,\mathbf{x}_2) \, + \, T^-(\omega_1^-) \, - \, T^-(\omega_2^-) \ = \ 0 \hspace{5mm} \mathrm{mod} \ m.
\end{equation}
But $\mathbf{y}$ does not feature at all in (\ref{redeq}).
\end{proof}

\begin{cor} \label{zk conj}
There exists $K \in \mathrm{Homeo}(\mathbb{S}^1)$ such that for $\mathbb{P}^-$-almost all $\omega^-$, $A_\mathbf{g}(\omega^-) = K(A_\mathbf{f}(\omega^-))$.
\end{cor}

\begin{proof}
We will use the notations $\nearrow$ and $\searrow$ as introduced before Lemma~\ref{pull}. Since the measure $\rho_{\mathbf{F}^{-1}}=R_{\mathbf{f}\ast}\mathbb{P}^+$ has full support, the projection of $\,\mathrm{supp}\,\mathfrak{r}\,$ under $\,(x,y) \mapsto x\,$ is the whole of $\mathbb{S}^1$. So let $K \colon \mathbb{S}^1 \to \mathbb{S}^1$ be a measurable function such that $\,\mathrm{graph}\,K \subset \mathrm{supp}\,\mathfrak{r}$. Let
\[ L := \{ x \in \mathbb{S}^1 \, : \, \exists \, x_n \!\nearrow\! x, \, y_n \!\searrow\! x \ \textrm{ s.t.\ } \lim_{n \to \infty} K(x_n) = \lim_{n \to \infty} K(y_n) = K(x) \}. \]
\noindent Since $\rho_\mathbf{F}$ is atomless, Lusin's theorem gives that $L$ is a $\rho_\mathbf{F}$-full measure set. Let $\hat{\Omega}^-:=\tilde{\Omega}^- \cap A_\mathbf{f}^{-1}(L)$. Since $L$ is a $\rho_\mathbf{F}$-full set, we have that $\hat{\Omega}^-$ is a $\mathbb{P}^-$-full set. Since $\rho_\mathbf{F}$ and $\rho_\mathbf{G}$ are atomless, given any $\omega^- \in \hat{\Omega}^-$, we have that for $(\mathbb{P}^- \otimes \mathbb{P}^-)$-almost all $(\tilde{\omega}_1^-,\tilde{\omega}_2^-) \in \hat{\Omega}^- \times \hat{\Omega}^-$,
\[ \{ \, (A_\mathbf{f}(\omega^-),A_\mathbf{g}(\omega^-)) \; , \; (A_\mathbf{f}(\tilde{\omega}_1^-),A_\mathbf{g}(\tilde{\omega}_1^-)) \; , \; (A_\mathbf{f}(\tilde{\omega}_2^-),A_\mathbf{g}(\tilde{\omega}_2^-)) \, \} \, \in \, \mathbf{S}_3. \]
\noindent For any $\omega^- \in \hat{\Omega}^-$, choosing $\tilde{\omega}_1^-,\tilde{\omega}_2^-$ as above, we have the following: By Lemma~\ref{redund}, either
\begin{align*}
K(]A_\mathbf{f}(\omega^-),A_\mathbf{f}(\tilde{\omega}_1^-)[) \ &\subset \ [A_\mathbf{g}(\omega^-),A_\mathbf{g}(\tilde{\omega}_1^-)] \ \textrm{ and} \\
K(]A_\mathbf{f}(\tilde{\omega}_1^-),A_\mathbf{f}(\omega^-)[) \ &\subset \ [A_\mathbf{g}(\tilde{\omega}_1^-),A_\mathbf{g}(\omega^-)],
\end{align*}
or
\begin{align*}
K(]A_\mathbf{f}(\omega^-),A_\mathbf{f}(\tilde{\omega}_1^-)[) \ &\subset \ [A_\mathbf{g}(\tilde{\omega}_1^-),A_\mathbf{g}(\omega^-)] \ \textrm{ and} \\
K(]A_\mathbf{f}(\tilde{\omega}_1^-),A_\mathbf{f}(\omega^-)[) \ &\subset \ [A_\mathbf{g}(\omega^-),A_\mathbf{g}(\tilde{\omega}_1^-)];
\end{align*}
\noindent so in either case, since $A_\mathbf{f}(\omega^-) \in L$, it follows that $K(A_\mathbf{f}(\omega^-)) \in \{A_\mathbf{g}(\omega^-),A_\mathbf{g}(\tilde{\omega}_1^-)\}$. And by the same reasoning, $K(A_\mathbf{f}(\omega^-)) \in \{A_\mathbf{g}(\omega^-),A_\mathbf{g}(\tilde{\omega}_2^-)\}$. Hence $K(A_\mathbf{f}(\omega^-))=A_\mathbf{g}(\omega^-)$.
\\ \\
It remains to show that $K$ is a homeomorphism. Let $\hat{L}:=A_\mathbf{f}(\hat{\Omega}^-)$. Note that since $\rho_\mathbf{F}$ and $\rho_\mathbf{G}$ have full support, $\hat{L}$ and $K(\hat{L})$ are dense in $\mathbb{S}^1$. Using Lemma~\ref{redund}, we obtain that either:
\begin{enumerate}[\indent (I)]
\item for any $x_1,x_2 \in \hat{L}$ and $x_3 \in \mathbb{S}^1$, if $\xi(x_1,x_2,x_3) \, := \, \{ \, (x_1,K(x_1)) \, , \, (x_2,K(x_2)) \, , \, (x_3,K(x_3)) \, \} \, \in \, \mathbf{S}_3$ then $\xi(x_1,x_2,x_3) \in \mathbf{V}$; or
\item for any $x_1,x_2 \in \hat{L}$ and $x_3 \in \mathbb{S}^1$, $\xi(x_1,x_2,x_3) \nin \mathbf{V}$.
\end{enumerate}
\noindent Assume that (I) holds. Fix any $x \in \mathbb{S}^1$ and any sequence $x_n \searrow x$ in $\hat{L}$ with $K(x_n)$ convergent, and suppose for a contradiction that $l:=\lim_{n\to\infty} K(x_n) \neq K(x)$; then since (I) holds, for each $n$ with $K(x_n) \neq K(x)$, we have that
\[ K(\hat{L} \, \cap \, [x_n,x]) \ \subset \ [K(x_n),K(x)], \]
\noindent and therefore,
\[ K(\hat{L}) \ \subset \ [l,K(x)], \]
\noindent contradicting the fact that $K(\hat{L})$ is dense. Hence, for any $x \in \mathbb{S}^1$ and any sequence $x_n \searrow x$ in $\hat{L}$, we have that $K(x_n) \to K(x)$. Similarly, for any $x \in \mathbb{S}^1$ and any sequence $x_n \nearrow x$ in $\hat{L}$, we have that $K(x_n) \to K(x)$. So since $\hat{L}$ is dense, it follows that $K$ is continuous everywhere. Let $K' \colon \mathbb{R} \to \mathbb{R}$ be a lift of $K$. Since (I) holds, $K'$ is monotone with $K'(t+1)=K'(t)+1$. Since $\rho_\mathbf{F}$ has full support and $\rho_\mathbf{G}$ is atomless, there cannot exist a non-empty open set on which $K$ is constant; hence $K'$ is strictly increasing. So $K$ is an orientation-preserving homeomorphism. A similar argument shows that if (II) holds, then $K$ is an orientation-reversing homeomorphism.
\end{proof}

\begin{cor}
$\mathbf{F}$ and $\mathbf{G}$ are deterministically topologically conjugate.
\end{cor}

\begin{proof}
Let $K$ be as in Corollary~\ref{zk conj}. We have that
\[ K \circ F_{\alpha_0}(\mathring{A}_\mathbf{f}(\omega)) \ = \ K(\mathring{A}_\mathbf{f}(\theta\omega)) \ = \ \mathring{A}_\mathbf{g}(\theta\omega) \ = \ G_{\alpha_0}(\mathring{A}_\mathbf{g}(\omega)) \ = \ G_{\alpha_0} \circ K(\mathring{A}_\mathbf{f}(\omega)) \]
for $\mathbb{P}$-almost all $\omega$; and so, since $\nu$ has full support on $\Delta$ and $\rho_\mathbf{F}$ has full support on $\mathbb{S}^1$, it follows that $K \circ F_\alpha=G_\alpha \circ K$ for all $\alpha$.
\end{proof}

\noindent So let $K \in \mathrm{Homeo}(\mathbb{S}^1)$ be a deterministic conjugacy from $\mathbf{F}$ to $\mathbf{G}$ fulfilling the description in Remark~\ref{assumption}. To complete the proof of Lemma~\ref{hardconj}(B), we split into two cases, the latter of which leads to a contradiction.

\subsubsection*{Case I: Either (a)~$K$ is orientation-preserving, or (b)~$m=2$ and $K$ is orientation-reversing.}

Let $\kappa$ be a homeomorphism commuting with $\tau_m$ such that $z_m(\kappa)=K$. For every $\alpha$ there exists $i \in \{0,\ldots,m-1\}$ such that
\[ f_\alpha \ = \ \tau_m^i \circ \kappa^{-1} \circ g_\alpha \circ \kappa. \]
But since $\Delta$ is connected, $i$ does not depend on $\alpha$. We need to show that $i = 0$. For $\mathbb{P}$-almost all $\omega$, let $n(\omega) \in \{0,\ldots,m-1\}$ be such that
\[ \kappa(a_\mathbf{f}(\omega)) \ = \ h_\omega(a_\mathbf{f}(\omega)) + [\tfrac{n(\omega)}{m}]. \]
A straightforward computation gives that $n(\theta\omega) \, = \, n(\omega) + i \ \mathrm{mod} \ m$; so since $\mathbb{P}$ is $\theta^m$-ergodic, we have that $n(\cdot)$ is constant $\mathbb{P}$-almost everywhere, and so $i=0$.

\subsubsection*{Case II: $m \geq 3$ and $K$ is orientation-reversing.}

For this case, we need to obtain a contradiction. Let $S \subset \{0,\ldots,m-1\}$ be the support of $T^-_{\ \ast}\mathbb{P}^-$.

\begin{lemma} \label{contradict}
There is an integer $p$ such that $S \subset \{p,p+1,p+2\}$ modulo $m$. In the case that $m=3$, $S$ has at most two elements.
\end{lemma}

\begin{proof}
By Lemma~\ref{vect}, we have that for $\mathbb{P}$-almost all $\omega$, $T^+(\omega^+)-T^-(\omega^-)$ is equal to either $-1$, $0$ or $1$ modulo~$m$; hence the first statement is clear. Now suppose that $m=3$. Then for $\mathbb{P}$-almost all $\omega$,
\begin{align*}
v( \, (A_\mathbf{f}(\omega^-),A_\mathbf{g}(\omega^-)) \, , \, (R_\mathbf{f}(\omega^+),R_\mathbf{g}(\omega^+)) \, ) \ &= \ v( \, (A_\mathbf{f}(\omega^-),K(A_\mathbf{f}(\omega^-))) \, , \, (R_\mathbf{f}(\omega^+),K(R_\mathbf{f}(\omega^+))) \, ) \\
&\in \ \{-1,1\} \hspace{5mm} \textrm{since $K([0])=[0]$.}
\end{align*}
So by Lemma~\ref{vect}, we have that for $\mathbb{P}$-almost all $\omega$, $T^+(\omega^+)-T^-(\omega^-)$ is equal to either $-1$ or $1$ modulo~$3$; hence we likewise have the second statement.
\end{proof}

\noindent Now \emph{fix} a point $\omega^-=(\alpha_i)_{i<0} \in \Omega^-$ with the property that for $\mathbb{P}^+$-almost all $\omega^+$, for all $n \in \mathbb{N}_0$, the following statements hold (where $\omega:=(\omega^+,\omega^-)$):
\begin{itemize}
\item $\varphi_\mathbf{F}(n,\omega)\mathring{A}_\mathbf{f}(\omega)=\mathring{A}_\mathbf{f}(\theta^n\omega)\ $ and $\ \varphi_\mathbf{G}(n,\omega)\mathring{A}_\mathbf{g}(\omega)=\mathring{A}_\mathbf{g}(\theta^n\omega)$;
\item $\mathring{A}_\mathbf{g}(\theta^n\omega)=K(\mathring{A}_\mathbf{f}(\theta^n\omega))=H_{\theta^n\omega}(\mathring{A}_\mathbf{f}(\theta^n\omega))$;
\item $S \, \ni \, \mathring{T}^-(\theta^n\omega) \, = \, \zeta(h_{\theta^n\omega}(\mathring{a}_\mathbf{f}(\theta^n\omega))) \ \mathrm{mod} \ m$;
\item $\varphi_\mathbf{f}(n,\omega) \, = \, h_{\theta^n\omega}^{-1} \circ \varphi_\mathbf{g}(n,\omega) \circ h_\omega$.
\end{itemize}
Let $x_\mathbf{f}:=a_\mathbf{f}(\omega^-)$ and let $x_\mathbf{g}:=a_\mathbf{g}(\omega^-)+[\frac{T^-(\omega^-)}{m}]$; so for $\mathbb{P}^+$-almost all $\omega^+ \in \Omega^+$, $h_\omega(x_\mathbf{f})=x_\mathbf{g}$. By \cite[Corollary~2.6]{Mal14}, as $n \to \infty$ the image measure of $\nu^{\otimes n}$ under $\boldsymbol{\alpha} \mapsto f_{\boldsymbol{\alpha}}^n(x_\mathbf{f})$ converges weakly to $\rho_\mathbf{f}$. So fix $N \in \mathbb{N}$ such that $I:=\overline{\{f_{\boldsymbol{\alpha}}^N(x) : \boldsymbol{\alpha} \in \Delta^N\}}$ has Lebesgue measure more than $1-\frac{1}{m}$. For convenience, for any $\boldsymbol{\alpha}=(\alpha_0,\ldots,\alpha_{N-1}) \in \Delta^N$, define $\theta_{\boldsymbol{\alpha}}^N\omega^- :=(\alpha_{i+N})_{i<0} \in \Omega^-$. A straightforward computation gives that for $(\nu^{\otimes N})$-almost all $\boldsymbol{\alpha} \in \Delta^N$,
\begin{equation} \label{T eq}
T^-(\theta_{\boldsymbol{\alpha}}^N\omega^-) \ = \ \zeta(g_{\boldsymbol{\alpha}}^N(x_\mathbf{g})) \ - \ \zeta(f_{\boldsymbol{\alpha}}^N(x_\mathbf{f})) \hspace{5mm} \mathrm{mod} \ m.
\end{equation}
Now for $(\nu^{\otimes N})$-almost all $\boldsymbol{\alpha} \in \Delta^N$,
\[ mg_{\boldsymbol{\alpha}}^N(x_\mathbf{g}) \ = \ A_\mathbf{g}(\theta_{\boldsymbol{\alpha}}^N\omega^-) \ = \ K(A_\mathbf{f}(\theta_{\boldsymbol{\alpha}}^N\omega^-)) \ = \ K(mf_{\boldsymbol{\alpha}}^N(x_\mathbf{f})). \]
\noindent So for \emph{every} $\boldsymbol{\alpha} \in \Delta^N$,
\[ mg_{\boldsymbol{\alpha}}^N(x_\mathbf{g}) \ = \ K(mf_{\boldsymbol{\alpha}}^N(x_\mathbf{f})). \]
\noindent So then, letting $\tilde{\kappa} \in \mathrm{Homeo}^-(\mathbb{S}^1)$ be such that $-\tilde{\kappa}$ commutes with $\tau_m$ and $z_m(-\tilde{\kappa})=-K$, we have that for all $\boldsymbol{\alpha} \in \Delta^N$ there exists $i \in \{0,\ldots,m-1\}$ such that
\[ g_{\boldsymbol{\alpha}}^N(x_\mathbf{g}) \ = \ \tilde{\kappa}(f_{\boldsymbol{\alpha}}^N(x_\mathbf{f})) + [\tfrac{i}{m}]. \]
\noindent But by the connectedness of $\Delta$, $i$ does not depend on $\boldsymbol{\alpha}$. So there exists $i \in \{0,\ldots,m-1\}$ such that, setting $\kappa:=\tau_m^i \circ \tilde{\kappa}$, for all $\boldsymbol{\alpha} \in \Delta^N$ we have
\[ g_{\boldsymbol{\alpha}}^N(x_\mathbf{g}) \ = \ \kappa(f_{\boldsymbol{\alpha}}^N(x_\mathbf{f})). \]
So by (\ref{T eq}), writing $y(\boldsymbol{\alpha}):=f_{\boldsymbol{\alpha}}^N(x_\mathbf{f})$, we have that
\[ \zeta(\kappa(y(\boldsymbol{\alpha}))) \ - \ \zeta(y(\boldsymbol{\alpha})) \, \in S \hspace{5mm} \mathrm{mod} \ m \]
\noindent for $(\nu^{\otimes N})$-almost all $\boldsymbol{\alpha} \in \Delta^N$. Note that the function $x \mapsto \zeta(\kappa(x))-\zeta(x)$ is piecewise-constant on $\mathbb{S}^1$, with the jumps in the value of this function being as follows:
\begin{itemize}
\item if $K([0]) \neq [0]$ then the jumps occur at $[\frac{k}{m}]$ and at $\kappa^{-1}([\frac{k}{m}])$ for all $k \in \{0,\ldots,m-1\}$, and modulo~$m$ these jumps have height $-1$ (with $x$ going anticlockwise);
\item if $K([0]) = [0]$ then the jumps occur at $[\frac{k}{m}]$ for all $k \in \{0,\ldots,m-1\}$, and modulo~$m$ these jumps have height $-2$ (with $x$ going anticlockwise).
\end{itemize}

\noindent But the range of $y(\cdot)$ on any $\nu^{\otimes N}$-full measure subset of $\Delta^N$ is dense in $I$, and $I$ is itself a connected set of Lebesgue measure more than $1-\frac{1}{m}$. Thus, the range of $\zeta(\kappa(y(\cdot)) \ - \ \zeta(y(\cdot))$ over a $\nu^{\otimes N}$-full measure set cannot be constrained to just three values if $m \geq 4$, and cannot be constrained to just two values if $m=3$. This contradicts Lemma~\ref{contradict}.

\bigskip

\noindent
\textbf{Acknowledgements. } The authors gratefully acknowledge financial support from the following sources: Russian Science Foundation grant 14-41-00044 at the Lobachevsky University of Nizhny Novgorod (JL), EU Marie-Curie IRSES Brazilian-European partnership in Dynamical Systems (FP7-PEOPLE-2012-IRSES 318999 BREUDS) (JL), EU Horizon 2020 Innovative Training Network CRITICS, grant no. 643073 (JL and MR), DFG grant CRC 701, Spectral Structures and Topological Methods in Mathematics (JN), EPSRC Doctoral Prize Fellowship (JN), EPSRC Career Acceleration Fellowship EP/I004165/1 (MR).

\sloppy


\end{document}